\documentclass[a4paper,11pt,bibliography=totoc,listof=totoc,headings=small]{scrartcl}
\pdfoutput=1
\let\chapter\section
\usepackage[style=alphabetic]{biblatex}
\bibliography{quellen.bib}
\usepackage[utf8]{inputenc}
\usepackage[english,ngerman]{babel}
\usepackage[babel]{csquotes}
\usepackage{amsmath}
\usepackage{amsfonts}
\usepackage{amssymb}
\usepackage{amsthm}
\usepackage{mathtools}
\usepackage{stmaryrd}
\usepackage{enumerate}
\usepackage{graphicx}
\usepackage[all]{xy}
\usepackage{color}
\usepackage{cancel}
\usepackage{dsfont}
\usepackage{faktor}
\usepackage{aliascnt}  
\usepackage{hyperref}
\usepackage{tikz}
\usepackage{xifthen}
\usetikzlibrary{arrows,decorations,backgrounds}
\usetikzlibrary{calc}
\usetikzlibrary{decorations.markings}
\usetikzlibrary{decorations.pathmorphing} 
\usepgflibrary{decorations.pathreplacing} 
\usepackage{framed}
\usepackage{pdflscape}
\usepackage[ruled,vlined]{algorithm2e}
\usepackage{fancyhdr}
\usepackage{pdfpages}
\usepackage{chngcntr}
\usepackage{wasysym}  
\usepackage{tocstyle}

\DeclareMathOperator{\GL}{GL}
\DeclareMathOperator{\SL}{SL}

\DeclareMathOperator{\Trans}{Trans}
\DeclareMathOperator{\Aff}{Aff}
\DeclareMathOperator{\Stab}{Stab}
\DeclareMathOperator{\sgn}{sgn}

\DeclareMathOperator{\sr}{sr}
\DeclareMathOperator{\concat}{concat}

\renewcommand{\div}{\mathrm{div}}

\def\quotient#1#2{%
    \raise0.5ex\hbox{$#1$}\big/\lower0.5ex\hbox{$#2$}%
}
\def\quotientr#1#2{%
    \lower0.5ex\hbox{$#1$}\big\backslash\raise0.5ex\hbox{$#2$}%
}

\def\smatr#1#2#3#4{%
  \bigl( \begin{smallmatrix} #1 & #2 \\ #3 & #4  \end{smallmatrix}\bigr)%
}

\def\matr#1#2#3#4{%
   \begin{pmatrix} #1 & #2 \\ #3 & #4  \end{pmatrix}%
}

\def\klvek#1#2{%
\bigl(\begin{smallmatrix} #1\\#2 \end{smallmatrix}\bigr)%
}

\newcommand{\PointL}[2]{\bigl(#1;#2\bigr)}

\newcommand{\Grenze}{35+24w}
\newcommand{\rp}{rational part}
\newcommand{\ip}{irrational part}
\newcommand{\gz}{root-looped 4-valent tree}

\newcommand{\EE}{\ensuremath{\mathbf{1}}}
\newcommand{\ZZ}{\ensuremath{\mathbb{Z}}}
\newcommand{\RR}{\ensuremath{\mathbb{R}}}
\newcommand{\CC}{\ensuremath{\mathbb{C}}}
\newcommand{\QQ}{\ensuremath{\mathbb{Q}}}
\newcommand{\HH}{\ensuremath{\mathbb{H}}}

\newcommand{\NN}{\ensuremath{\mathbb{N}}}
\newcommand{\inv}{\ensuremath{^{-1}}}
\newcommand{\LapC}{\Delta_G}
\newcommand{\om}{w}
\newcommand{\PN}[1][]{\mathcal{P}_{{N#1}}}

\definecolor{grau}{gray}{0.8}

\graphicspath{ {./Bilder/} }

\begin{document}
\setcounter{tocdepth}{2}
\selectlanguage{english}
\title{On the Critical Exponent of Infinitely Generated Veech Groups}
\author{Ralf Lehnert}
\date{}
\maketitle
\vspace*{-3em}
\begin{abstract}\small
We prove the existence of Veech groups having a critical exponent strictly greater than any elementary Fuchsian group (i.e.\@ $>\frac{1}{2}$) but strictly smaller than any lattice (i.e.\@ $<1$). More precisely, every affine covering of a primitive L-shaped Veech surface $X$ ramified over the singularity and a non-periodic connection point $P\in X$ has such a Veech group. Hubert and Schmidt (\cite{HSInf}) showed that these Veech groups are infinitely generated and of the first kind. We use a result of Roblin and Tapie (\cite{RobTap}) which connects the critical exponent of the Veech group of the covering with the Cheeger constant of the Schreier graph of $\SL(X)/\Stab_{\SL(X)}(P)$. The main task is to show that the Cheeger constant is strictly positive, i.e.\@ the graph is non-amenable. In this context, we introduce a measure of the complexity of connection points that helps to simplify the graph to a forest for which non-amenability can be seen easily.
\end{abstract}
\tableofcontents

\newtheorem{mthm}{Theorem}
\renewcommand*{\themthm}{\Alph{mthm}}
\theoremstyle{definition}
\newtheorem{defi}{Definition}[section]
\newaliascnt{theorem}{mthm}
\newaliascnt{theoremc}{defi}
\newaliascnt{proposition}{defi}
\newaliascnt{kproposition}{defi}
\newaliascnt{lemma}{defi}
\newaliascnt{remark}{defi}
\newaliascnt{corol}{defi}
\newaliascnt{conjecture}{defi}
\newaliascnt{example}{defi}
\newaliascnt{convention}{defi}
\newaliascnt{question}{defi}
\newtheorem{exa}[example]{Example}
\newtheorem{conv}[convention]{Convention}
\newtheorem{rem}[remark]{Remark}
\theoremstyle{plain}
\newtheorem{thm}[theorem]{Theorem}
\newtheorem{thmD}{Theorem}
\newtheorem{propD}[thmD]{Proposition}
\newtheorem{lemD}[thmD]{Lemma}
\newtheorem{thmc}[theoremc]{Theorem}
\newtheorem{prop}[proposition]{Proposition}
\newtheorem{kprop}[kproposition]{Key Proposition}
\newtheorem{lem}[lemma]{Lemma}
\newtheorem{cor}[corol]{Corollary}
\newtheorem{ques}[question]{Question}
\newtheorem{conj}[conjecture]{Conjecture}

\providecommand*{\mthmautorefname}{Main Theorem}
\providecommand*{\thmautorefname}{Theorem}
\providecommand*{\thmDautorefname}{Theorem}
\providecommand*{\theoremcautorefname}{Theorem}
\providecommand*{\theoremautorefname}{Theorem}
\providecommand*{\propositionautorefname}{Proposition}
\providecommand*{\kpropositionautorefname}{Key Proposition}
\providecommand*{\lemmaautorefname}{Lemma}
\providecommand*{\remarkautorefname}{Remark}
\providecommand*{\defiautorefname}{Definition}
\providecommand*{\corolautorefname}{Corollary}
\providecommand*{\exampleautorefname}{Example}
\providecommand*{\conventionautorefname}{Convention}
\providecommand*{\questionautorefname}{Question}
\def\algorithmautorefname{Algorithm}
\def\sectionautorefname{Section}%
\def\subsectionautorefname{Section}%
\def\subsubsectionautorefname{Section}

\addsec{Introduction}
Translation surfaces are Riemann surfaces $X$ that consist of a finite set of euclidean polygons glued together at parallel sides by translations. The \emph{affine group} $\Aff(X)$ of $X$ consists of all orientation-preserving self-homeomorphisms $f: X \rightarrow X$ that map the set of vertices $S(X)$ of the polygons (also called \emph{singularities}) to itself and are locally affine on $X-S(X)$. We consider only connected translation surfaces of finite volume, thus the linear part -- i.e. the derivative -- of $f$ is a globally constant $2\times 2$-matrix $Df=A$ of determinant $1$. The image of the affine group under the derivation map is the \emph{Veech group} $\SL(X)$. The Veech group is a discrete subgroup of $\SL_2(\RR)$ and thus a Fuchsian group.

There is a rough sizing of Fuchsian groups distinguishing them by the size of their limit set $\Lambda$: elementary Fuchsian groups and non-elementary ones which are again subdivided into those of the first and those of the second kind. The ``largest'' Fuchsian groups are \emph{lattices}, finitely generated of the first kind, followed by the infinitely generated ones of the first kind.

A refinement of this classification is to additionally consider the \emph{critical exponent} $\delta(\Gamma)$, defined as the infimum of all $a\in \RR$ such that the \emph{Poincaré series} $\sum_{\gamma\in\Gamma}\exp(-a\rho_\HH(i,\gamma\circ i))$ converges. By \cite{BJ} the critical exponent also is the Hausdorff dimension of the conical limit set. For all infinite Fuchsian groups the bounds $0\leq \delta \leq 1$ hold.

The main result of this paper is the following:
\begin{thm}\label{mainresult}
There exist translation surfaces whose Veech groups have critical exponent strictly between $\frac{1}{2}$ and $1$.
\end{thm}

Since the critical exponent of elementary groups is at most $\frac{1}{2}$, these are excluded as candidates for the main theorem. So are lattices because they always have critical exponent $1$. Until 2003, lattices were the only known non-elementary Veech groups. Translation surfaces with lattice Veech groups are called \emph{Veech surfaces} and are of special interest because they satisfy the Veech dichotomy (\cite{Vee}). Calta (\cite{Cal}) constructed Veech surfaces of genus $2$ that are primitive, i.e. do not arise via coverings and McMullen (\cite{McMDiscr} and \cite{McMTor}) classified all primitive Veech surfaces of genus $2$ and showed that they -- up to the action of $\GL_2(\RR)$ -- have the shape of an $L$ with suitable side lengths.

While there are still no translation surfaces known with a Veech group of the second kind, McMullen (\cite{McMInfCo}) and independently Hubert and Schmidt (\cite{HSInf}) found ways to construct translation surfaces with infinitely generated Veech groups of the first kind. Since then, estimating the critical exponent of these groups has been an open question. We will concentrate on the construction of Hubert and Schmidt and give an answer to this question in this paper.

Points $P\in X$ with the property that all geodesic rays emanating from a singularity and passing through $P$ eventually end in a singularity are called \emph{connection points}. These play an important role in the construction of Hubert and Schmidt. They construct translation surfaces whose Veech group is commensurable to the stabilizer $\SL(X;P)$ of $P$ and show that if $P$ is non-periodic, i.e. has an infinte orbit under the action of $\SL(X)$, then this group is infinitely generated and of the first kind.

\autoref{mainresult} will follow from
{
\renewcommand{\theproposition}{\ref{mainprop}}
\begin{prop}
For every non-periodic connection point $P$ on a primitive L-shaped Veech surface $X$ the critical exponent of the (infinitely generated) stabilizer subgroup $\SL(X;P):=\Stab_{\SL(X)}(P)$ is strictly between $\frac{1}{2}$ and $1$.
\end{prop}
\addtocounter{defi}{-1}
}

A natural question in this context is the dependence on $P$: every point $P'$ in the $\SL(X)$-orbit of $P$ also is a non-periodic connection point and the groups $\SL(X;P)$ and $\SL(X;P')$ are conjugate and thus have the same critical exponent. To illustrate the dependence on the point, we focus on a particular L-shaped surface $L_8$ and establish some bounds on the number of different orbits.

The connection points are the points of the form $P=(x_r+x_i\sqrt{2};y_r+y_i\sqrt{2})$ with $x_r,x_i,y_r,y_i\in\QQ$. The least common denominator $N(P)$ of the four reduced fractions is an invariant for the orbit of $P$. For every $N\in\mathbb{N}$ we define $\PN$ as the set of connection points $P$ with $N(P)=N$. Since there is no upper bound on the denominators of connection points there are infinitely many different orbits. But for fixed $N$ we can show:

\begin{thm}\label{thm:Orbits}
Set $\om\coloneqq \sqrt{2}$ and fix $N\in\NN$. The set $\PN$ decomposes into a \textbf{finite} number of $\SL(L_8)$-orbits.
\end{thm}

This is shown in \autoref{sect:orbits}. \autoref{sect:concl} contains the proof of \autoref{mainresult}. For the proof we need a way to measure the complexity of a connection point $P$. This is done by the value $s(P)\coloneqq |x_i|+|y_i|$ introduced and analyzed in \autoref{sect: technLemm}. The section consists of some technical lemmata describing the effect of particular Veech group elements on $s(P)$. \autoref{critexp} describes some background about the critical exponent and the connection between the critical exponent and the Laplacian on hyperbolic manifolds. In this context, the amenability of Schreier graph appears. \autoref{sect:Schreier} explains this term and provides tools to prove non-amenability of (Schreier) graphs. The paper starts with some background information on translation surfaces and Veech groups in \autoref{sect:TSVG}.

The author thanks the European Research Council ERC-StG 257137 for financial support as well as J.\ Cuno, P.\ Hubert , M.\ Möller and S.\ Tapie for useful discussions.


\section{Background on Translation Surfaces and Veech Groups}\label{sect:TSVG}
\subsection{Basic Definitions and Observations}
\emph{Translation surfaces} $X$ can be viewed as Riemann surfaces together with an atlas such that all transition maps are translations away from a discrete subset -- the \emph{singularities}. This is equivalent to a choice of a holomorphic one-form $\eta$ on $X$. The third way to define a translation surface -- which we will mostly use -- is as a finite collection of euclidean polygons glued along parallel sides by translations. For details of the equivalence of these definitions cf.\@ e.g.\@ the survey \cite{Mas}.

We observe and define the following: the total angle around each vertex $v$ of a polygon is $2\pi m_v$ with $m_v$ a positive integer. The vertex $v$ is called a \emph{singularity} and the number $m_v-1$ its \emph{order}. The set of all singularities of the translation surface $X$ is denoted by $S(X)$. A geodesic emanating from a singularity is called a \emph{separatrix}. A geodesic, without singularities in its interior, which connects two singularities is called a \emph{saddle connection}. In this context we can define a special type of points on translation surfaces that are important in the construction of translation surfaces with infinitely generated Veech groups developed by Hubert and Schmidt (\cite{HSInf}):

\begin{defi}
A nonsingular point $P$ is a \emph{connection point} of a translation surface $X$ if every separatrix without singularities in its interior that passes through $P$ can be extended to a saddle connection.
\end{defi}

By the Gau\ss -Bonnet theorem the genus of the translation surface $X$ can be calculated from the equation $\sum_{v\in S(X)}(m_v-1)=2g-2$. The space $\Omega M_g$ of all pairs $(X,\eta)$ with $X$ a Riemann surface of genus $g$ and $\eta$ a holomorphic one-form can be stratified by the orders of the singularities. For $g=2$ there are two strata: $\Omega M_2(1,1)$ -- the translation surfaces with two singularities of order $1$ -- and  $\Omega M_2(2)$ -- the translation surfaces with a single singularity of order $2$. The stratification is invariant under the action of $\GL_2(\RR)$ on the space of translation surfaces that is induced by the the action on the plane and thus on the polygons by linear maps.

Let $X$ be a connected translation surface of finite volume. The \emph{translation group} $\Trans(X)$ consists of all orientation-preserving self-homeomorphisms $f: X \rightarrow X$, which are locally translations and map $S(X)$ to itself. Accordingly, the \emph{affine group} $\Aff(X)$ consists of all orientation-preserving self-homeomorphisms $f: X \rightarrow X$, which are locally affine and map $S(X)$ to itself. Note that, since $X$ is connected and of finite volume, the linear part -- i.e. the derivative -- of $f\in\Aff(X)$ is a globally constant $2\times 2$-matrix $Df=A$ of determinant $1$. The \emph{Veech group} $\SL(X)$ is the image of $\Aff(X)$ in $\SL_2(\RR)$ under the derivation map. Another way to see the Veech group is as the stabilizer of the translation surface under the action of $\GL_2(\RR)$ described above.

These three groups form a short exact sequence: $\Trans(X) \hookrightarrow \Aff(X) \twoheadrightarrow \SL(X)$. By Proposition 4.4 of \cite{HL} in genus $2$ the translation group is trivial, so the affine group and the Veech group are isomorphic and we identify elements of the Veech group with their preimage in the affine group. For $\gamma\in\SL(X)$ and $P\in X$ we write $\gamma\circ P$ for the image of $P$ under $\gamma$.

The Veech group of a translation surface is a discrete subgroup of $\SL_2(\RR)$ (cf.\@ e.g.\@ Section 1.3 of \cite{HSIntro}). Hence it is a Fuchsian group, i.e. it acts properly discontinuously on the (hyperbolic) upper half plane $\mathbb{H}$ by Moebius transformations. Fuchsian groups are a well-studied subject (cf.\@ e.g.\@ \cite{BeaBook} and \cite{Kat}) and we will now recall some important notions.

Fix a point of $\mathbb{H}$ -- we choose the point $i$ -- as a base point and let $\Gamma$ be a Fuchsian group. Note that all accumulation points of the orbit $\Gamma i$ have to be at the boundary $\partial\HH=\RR\cup \{\infty\}$, because otherwise the action would not be properly discontinuous. The set $\Lambda$ of all accumulation points $r\in \partial \HH$ of the orbit $\Gamma i$ is called the \emph{limit set}. If $\Lambda$ is finite, $\Gamma$ is called \emph{elementary}, otherwise \emph{non-elementary}. Non-elementary groups whose limit set is the whole boundary are called Fuchsian groups \emph{of the first kind}; the other non-elementary groups are \emph{of the second kind}. If $\Gamma$ has a convex fundamental domain with finitely many sides, it is called \emph{geometrically finite}, if it has a fundamental domain, which has finite hyperbolic area, it is called \emph{of finite covolume} or \emph{lattice}.

Two Fuchsian groups $\Gamma$ and $\Gamma'$ are said to be \emph{commensurate}, if they have a commom subgroup of finite index in both $\Gamma$ and $\Gamma'$. They are called \emph{commensurable}, if they have subgroups of finite index, which are conjugate by an element of $\SL_2(\RR)$.

The following lemma is a consequence of Theorem 4.5.1 and Theorem 4.6.1 of \cite{Kat} and plays a crucial role in the construction of infinitely generated Veech groups by McMullen as well as in the construction by Hubert and Schmidt:
\begin{lem}[\cite{HSInf}, Lemma 3] A Fuchsian group of the first kind either is a lattice or it is infinitely generated.\label{lem:LattOrInfGen}
\end{lem}

Another important result comparing different Veech groups is the following by Gutkin and Judge:
\begin{prop}[\cite{GutJud}, Theorem 4.9]\label{prop:commens} Let $p : Y \rightarrow X$ be an affine covering of translation surfaces. Then the groups $\SL(Y)$ and $\SL(X)$ are commensurable. If $p$ is a translation covering, they are commensurate.
\end{prop}

\subsection{Veech Groups of Different ``Sizes''}\label{Sizes}
Since Veech groups are discrete subgroups of $\SL_2(\RR)$ and thus Fuchsian groups, it is a natural question, which types or sizes of Fuchsian groups appear as Veech groups. By Theorem $1.1$ of \cite{MMAffGr} the Veech group of a generic translation surface $X$ of genus $g\geq 2$ is very small, namely isomorphic to $\quotient{\ZZ}{2\ZZ}$ or trivial, depending on whether $X$ belongs to a hyperelliptic component of its stratum or not.

The next larger groups are cyclic groups. Whereas in every stratum there exists a translation surface, whose Veech group is cyclically generated by a parabolic element (\cite{MMAffGr}, Proposition 1.4), it is still an open question, if there exists a translation surface with Veech group cyclically generated by a hyperbolic element.

There are no translation surfaces known which have a Veech group that is non-elementary and of the second kind.

\subsubsection{Lattices}\label{subs:Lattices}
The best-studied translation surfaces are the ones having a lattice Veech group. They satisfy the Veech dichotomy (\cite{Vee}) and therefore are called \emph{Veech surfaces}. In particular, the Veech surfaces of genus $2$ are well-known. In \cite{Cal} Calta constructs Veech surfaces of genus $2$ that are \emph{primitive}, i.e. not arising via a covering construction. McMullen classified all primitive Veech surfaces of genus $2$ by specifying prototypes such that each primitive Veech surface is in the $\GL_2(\RR)$-orbit of one of these prototypes.

\begin{defi}\label{def:LDorLDeps}
Let $D\geq 5$ be an integer $\equiv 0$ or $1\mod 4$ and not a square. 
\begin{itemize}
\item[a)] If $D\equiv 0 \mod 4$ we set $w\coloneqq\sqrt{\frac{D}{4}}$ and define $L_D$ to be the translation surface obtained from the $L$-shaped polygon with side lengths and identifications as shown in the upper left picture of \autoref{pic:LDLDmLDp}.
\item[b)] If $D\equiv 1 \mod 4$ we set $w\coloneqq\frac{1+\sqrt{D}}{2}$ and define $L_{D,-1}$ to be the translation surface obtained from the $L$-shaped polygon with side lengths and identifications as shown in the upper right picture of \autoref{pic:LDLDmLDp}.
\item[c)] Additionaly, if $D\equiv 1 \mod 8$ we set $w\coloneqq\frac{1+\sqrt{D}}{2}$ and define $L_{D,+1}$ to be the translation surface obtained from the $L$-shaped polygon with side lengths and identifications as shown in the lower picture of \autoref{pic:LDLDmLDp}.
\end{itemize}
\end{defi}

\begin{figure}[ht]
\begin{center}
\includegraphics[width=0.92\linewidth]{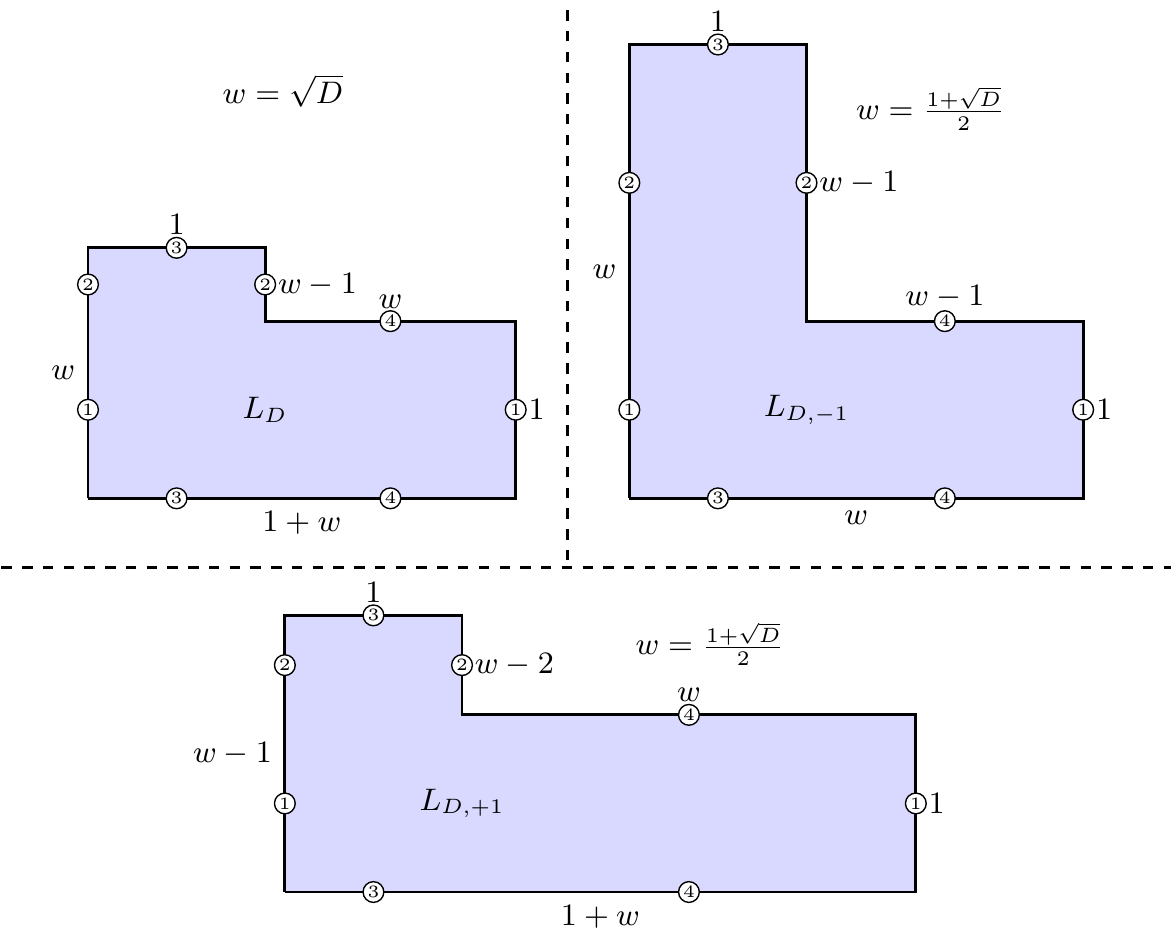}
\caption{The surfaces $L_D$, $L_{D,-1}$, and $L_{D,+1}$.}
\label{pic:LDLDmLDp}
\end{center}
\end{figure}

These surfaces are the prototypes of McMullen's classification of primitive Veech surfaces in $\Omega M_2(2)$:
\begin{prop}[\cite{McMDiscr}, Corollary 1.3]\label{prop:ClassificationVS}
The translation surfaces $L_D$, $L_{D,-1}$, and $L_{D,+1}$ are Veech surfaces of different $\GL_2(\RR)$-orbits. Every primitive Veech surface $X$ of $\Omega M_2(2)$ is in the orbit of one of these for a suitable $D$.
\end{prop}

The number $D$ in the above theorem is the discriminant of the \emph{trace field} of $X$, the number field $\QQ$ adjoint all traces of elements of the Veech group $\SL(X)$.

McMullen completed the classification of primitive Veech surfaces of genus $2$ with the following theorem:
\begin{prop}[\cite{McMTor}, Theorem 1.1]\label{prop:ClassificationVS11}
All primitive Veech surfaces in $\Omega M_2(1,1)$ are in the $\GL_2(\RR)$-orbit of the surface obtained from a regular decagon by gluing opposite sides.
\end{prop}

The surfaces $L_D$, $L_{D,-1}$, and $L_{D,+1}$ are an important ingredient of \autoref{mainprop}, which implies our main result. Since we will prove many statements holding for all these surfaces we introduce the notion $L_{D,\epsilon}$ to summarize the surfaces of the three types.

We need to understand which points are \emph{periodic points} -- i.e. have a finite orbit under the action of the Veech group -- and which points are connection points.

\begin{prop}[\cite{MMPerPoi}, Theorem 5.1] \label{thm:periodic} The only periodic points on a primitive Veech surface in $\Omega M_2(2)$ are the $6$ fixed points of the hyperelliptic involution.
\end{prop}

From now on, we choose the coordinates of points on the $L$-shaped surfaces such that the origin lies in the lower left corner. Moreover, for the identified points on the sides of the polygon, we choose the point with smaller coordinates. 

To find the connection points, we have to look at the \emph{periods} of the surfaces: to define them, let $\eta$ be the holomorphic one-form corresponding to $X$. The \emph{group of periods} of $X$ with respect to $\eta$ is the additive group $\{\int_\pi\eta \mid \pi\in H_1(X,\ZZ)\}<\CC\cong\RR^2$. All sides of the surface $L_{D,\epsilon}$ are horizontal or vertical and their lengths are in $\QQ(w)$. Thus the periods are contained in $\QQ(w)^2$. Since there is just one singularity, every saddle connection vector is also a period and hence also contained in $\QQ(w)^2$.

Furthermore, by Theorem A.1 from \cite{McMInfCo} the set of periodic directions of $L_{D,\epsilon}$ is precisely $\mathbb{P}^1(\QQ(w))$. In Section 3.2 of \cite{HSInf} translation surfaces with these properties are called \emph{of strong holonomy type}. It is shown there that exactly the points with both coordinates in the field $\QQ(w)$ are connection points. Thus we obtain:

\begin{prop}\label{prop: ConnPoints}
The connection points of $L_D$ are the points $P=\PointL{x_r+x_iw}{y_r+y_iw}$ with $w=\sqrt{\frac{D}{4}}$ and $x_r,x_i,y_r,y_i\in\QQ$.

The connection points of $L_{D,\pm 1}$ are the points $P=\PointL{x_r+x_iw}{y_r+y_iw}$ with $w=\frac{1+\sqrt{D}}{2}$ and $x_r,x_i,y_r,y_i\in\QQ$.
\end{prop}

To distinguish connection points we define the following:
\begin{defi}\label{defi:NPN}
For $P=\PointL{x_r+x_iw}{y_r+y_iw}$ with $x_r,x_i,y_r,y_i\in\QQ$ reduced fractions we define \emph{$N(P)$} to be the least common denominator of $x_r,x_i,y_r$ and $y_i$. Furthermore, we denote the set of all connection points $P$ with fixed $N(P)=N$ by \emph{$\PN$}.
\end{defi}

In fact, in our cases the periods are not only contained in $\QQ(w)$ but even in $\ZZ[w]$. This implies that all Veech group element entries are elements of $\ZZ[w]$. Moreover, if $\gamma$ is any element of the Veech group and $P\in \PN$ is any connection point of $L_{D,\epsilon}$, then the coordinates of $\gamma \circ P$ equal the coordinates of $\gamma \cdot P$ up to some period. This implies an important proposition concerning $N(P)$:
\begin{prop}\label{thm: orbitssameN}
The value $N(P)$ is an invariant for the orbit $\SL(L_{D,\epsilon})\circ P$, i.e. for $P\in\PN$ the whole orbit $\SL(L_{D,\epsilon})\circ P$ is contained in $\PN$.
\end{prop}

\subsection{Infinitely Generated Veech Groups}
Veech surfaces are not the only known translation surfaces that have a non-elementary Veech group. At the beginning of this millennium McMullen and independently Hubert and Schmidt could prove the existence of translation surfaces with infinitely generated Veech groups of the first kind. We will briefly describe the two different approaches, but concentrate on the second one.

In \cite{McMInfCo}, Theorem 10.1, McMullen shows that for translation surfaces of genus $2$ the limit set of the Veech group is either empty, a single point or the whole boundary $\partial \HH$. In particular, if the Veech group contains a hyperbolic element, the group is of the first kind, so the Veech group is infinitely generated if it is not a lattice.

Every non-primitive Veech surface of genus $2$ lies in the orbit of a square-tiled surface and for these surfaces all elements of the Veech group have rational trace. Together with the classification of primitive Veech surfaces of genus $2$ (\autoref{prop:ClassificationVS11}), this implies:

\begin{prop}[\cite{McMTor}, Theorem 1.3] Every translation surface $X\in\Omega M_2(1,1)$ which contains a hyperbolic element with irrational trace either is in the $\GL_2(\RR)$-orbit of the regular decagon surface or has an infinitely generated Veech group.
\end{prop}

\cite{McMInfCo} also gives concrete examples of such surfaces.

Additionally, Hubert and Schmidt constructed translation surfaces, such that the Veech group is of the first kind, but not a lattice. We sketch this construction from \cite{HSInf} which we will use during this paper.

\begin{defi}
Given a translation surface $X$ with singularities $S(X)$ and a non-singular point $P$, define the \emph{marking of $X$ at $P$} as a new translation surface $(X;P)$ by adding $P$ to the set of singularities. Let the group of affine diffeomorphisms of $(X;P)$ be the subgroup of $\Aff(X)$ consisting of the maps that fix $P$. Accordingly, define the Veech group $\SL(X;P)$ as the stabilizer subgroup of $P$ in $\SL(X)$.
\end{defi}

As we will see in \autoref{rem:vertorbit}, there is a bijection between the right cosets of $\SL(X)$ modulo $\SL(X;P)$ and the orbit points of $P$ under the action of $\SL(X)$. Thus the non-periodicity of $P$ guarantees that $\SL(X;P)$ is of infinite index in $\SL(X)$ and hence not a lattice.

To show that $\SL(X;P)$ is of the first kind, Hubert and Schmidt use Proposition 3.1 of \cite{Vor}, which states that the set of directions of geodesic segments emanating from $P$ and encountering a singularity is dense in $S^1=\partial\HH$. If $P$ is a connection point, this means the set of directions of saddle connections through $P$ is dense in $S^1$. To each of these saddle connections belongs a parabolic element of $\SL(X)$ fixing the saddle connection pointwise, in particular fixing $P$. A group containing a parabolic element with eigenvector $\klvek{x}{y}$ has the fixed direction $\frac{y}{x}$ in its limit set. Hence the limit set of $\SL(X;P)$ is dense in $S^1$, which is not possible for groups of the second kind (cf.\@ e.g.\@ Theorem 3.4.6 of \cite{Kat}). This yields

\begin{prop}[\cite{HSInf}, Proposition 1] Let $P$ be a non-periodic connection point on a Veech surface $X$. Then $\SL(X;P)$ is infinitely generated and of the first kind.
\end{prop}

By \autoref{prop:commens} the Veech groups of affine coverings of $(X;P)$ are commensurable to $\SL(X;P)$. Thus they are also infinitely generated and it remains to find Veech surfaces with non-periodic connection points. But we have seen in \autoref{prop: ConnPoints} that the Veech surfaces in $\Omega M_2(2)$ have infinitely many connection points; thus the following proposition holds and gives us candidates for Veech surfaces of the first kind with critical exponent strictly smaller than $1$ (for the definition of the critical exponent see \autoref{critexp}):

\begin{prop}
An affine covering of the Veech surface $L_{D,\epsilon}$ ramified over the singularity and a point $P\in\PN$ for any $N$ has an infinitely generated Veech group of the first kind.
\end{prop}


\section{Amenability of Graphs and the Combinatorial Laplacian}\label{sect:Schreier} 
In this section we introduce and analyze the Cheeger constant of graphs and the concept of Schreier graphs. As we will see in \autoref{critexp} to estimate the Cheeger constant of Schreier graphs will be the main task of this paper. Moreover, we describe the bottom of the spectrum of the combinatorial Laplacian on a graph and the connection to the Cheeger constant in the second subsection.
\subsection{Amenability of Graphs}
\subsubsection{The Cheeger Constant}
Let $G=(V,E)$ be a graph with vertex set $V$ and edge set $E$. For any subset $M\subset V$ we define the (vertex-) \emph{boundary} $\partial M\subset M$ to consist of the vertices of $M$ that have a neighbor in $V-M=M^c$. Accordingly, the \emph{interior} is $\mathring{M}:=M-\partial M$.
 \begin{defi}\label{def:cheegerconst}
 Let $G=(V,E)$ be a graph. For a finite nonempty set of vertices $M\subset V$ we define $c(M):=\frac{|\partial M|}{|M|}$. The \emph{Cheeger constant} of $G$ is \[\mathfrak{c}(G):=\inf_{\textnormal{finite } M\subset V} c(M).\] If $\mathfrak{c}(G)=0$, the graph $G$ is called \emph{amenable}; otherwise it is \emph{non-amenable}. 
  \end{defi}

In the following we will be interested in Cheeger constants of Schreier graphs, which are directed graphs that may contain loops and multi-edges. But since these do not affect the boundary of a vertex subset and thus the Cheeger constant, we can and will forget the directions of the edges, omit loops and consider just one edge for every multi-edge -- so we will deal with \emph{simple graphs}.

Unfortunately there are not many results on (non-)amenability of graphs. As described and used in \cite{Kap} the main technique is the following due to Bartholdi. He generalizes a result of Grigorchuk about Cayley graphs (\cite{Gri}) to all regular graphs:

\begin{prop}[\cite{Bar}] Let $G=(V,E)$ be a connected $d$-regular graph, choose a point $v_0\in V$ and let $a_n$ be the number of reduced edge-paths of length $n$ from $v_0$ to $v_0$. Then $G$ is amenable if and only if $\limsup_{n\rightarrow \infty} \sqrt[n]{a_n}=d-1$.
\end{prop}

Schreier graphs are regular, but the groups that come into play are too complicated to count the (reduced) edge-paths from the base point to itself. Hence we want to find another (more elementary) method to prove non-amenability of a Schreier graph. Let us begin by collecting some facts about Cheeger constants that will help us:

First it is obvious that a graph containing a finite connected component $C$ is amenable since $|\partial C|=0$. But on the other hand we have 
\begin{prop}\label{PerPoiWegl}
If $G=(V,E)$ is an amenable graph without any finite connected component, then for any finite subset $F\subset V$ the subgraph induced by $V'=V - F$ is amenable, too. 
\end{prop}

\begin{proof}
Since $G$ is amenable, there exists a sequence of finite sets $M_i\subset V$ with $c(M_i)\rightarrow 0$ as $i\rightarrow \infty$. There is no finite connected component, hence each $M_i$ has at least one boundary vertex, i.e.\@ $|\partial M_i|\geq 1$. Thus $|M_i|$ must tend to $\infty$ to permit $c(M_i)=\frac{|\partial M_i|}{|M_i|}\rightarrow 0$.

Now let us set $M_i':=M_i - P=M_i\cap V'$ and look at the sequence $(c(M_i'))$:
\[c(M_i')=\frac{|\partial M_i'|}{|M_i'|}\leq \frac{|\partial M_i|}{|M_i|-|F|}.\]
The reciprocal of this term is $\frac{|M_i|-|F|}{|\partial M_i|}=\frac{|M_i|}{|\partial M_i|}-\frac{|F|}{|\partial M_i|}.$ The minuend tends to infinity, the subtrahend is bounded above by $|F|$ and hence $c(M_i')\rightarrow 0 \textnormal{ as } i\rightarrow \infty.$ Thus the subgraph induced by $V'$ is amenable.
\end{proof}

\begin{prop}
\label{MMR}
Let $G=(V,E)$ be graph with connected components $K_i$ and let $M$ be a finite subset of the vertex set $V$. Setting $M_i\coloneqq M\cap K_i$ the following holds: \[ \min_{i\mid M_i\neq \emptyset} \frac{|\partial M_i|}{|M_i|} \leq \frac{|\partial M|}{|M|} \leq \max_{i\mid M_i\neq \emptyset} \frac{|\partial M_i|}{|M_i|}. \]
\end{prop}

\begin{proof}
As the $M_i$ are in different connected components the boundary $\partial M_i$ is the intersection of $M_i$ with $\partial M$.
First consider the case of two components, so $M$ is the disjoint union $M_1 \sqcup M_2$, and $\partial M=(\partial M \cap M_1) \sqcup (\partial M \cap M_2)=\partial M_1\sqcup\partial M_2$. Without loss of generality we assume $c(M_1)=\frac{|\partial M_1|}{|M_1|}\leq \frac{|\partial M_2|}{|M_2|}=c(M_2)$. This is equivalent to $|\partial M_1|\leq \frac{|M_1|\cdot |\partial M_2|}{|M_2|}$. Thus we obtain \[c(M)=\frac{|\partial M|}{|M|}=\frac{|\partial M_1|+|\partial M_2|}{|M_1|+|M_2|}\leq \frac{\frac{|M_1|\cdot |\partial M_2|+|M_2|\cdot|\partial M_2|}{|M_2|}}{|M_1|+|M_2|}=\frac{|\partial M_2|}{|M_2|}=c(M_2).\]
By analogy, one gets $c(M)\geq c(M_1)$. Note that $M$ is finite and thus only finitely many $M_i$ are non-empty. Since $\min\{a,b,c\}=\min\{\min\{a,b\},c\}$, by induction we obtain the general case of arbitrary many connected components.
\end{proof}

\begin{cor}\label{connected}
Let $G$ be a graph with connected components $K_i$. Then \[\mathfrak{c}(G)=\inf_{i}\mathfrak{c}(K_i).\]
\end{cor}

\begin{proof}
That $\mathfrak{c}(G)\leq \inf_{i}\mathfrak{c}(K_i)$ is clear by definition of the Cheeger constant.

To show $\mathfrak{c}(G)\geq\inf_{i}\mathfrak{c}(K_i)$ let $M$ be an arbitrary (non-empty) finite set of vertices and $M_i:=M\cap K_i$ as in \autoref{MMR}. Thus $c(M_i)\leq c(M)$ for at least one $i$.

Now let $(M_j)_{j\in \mathbb{N}}$ be a sequence of finite vertex sets with $\inf_j c(M_j) = \mathfrak{c}(G)$ and let $M_{j,i}$ be $M_j\cap K_i$. Since for all $j$ there exists an $i$ such that $c(M_j)\geq c(M_{j,i})$, this implies \[\mathfrak{c}(G)=\inf_j c(M_j) \geq \inf_j \inf_i c(M_{j,i}) = \inf_{i,j} c(M_{j,i}) = \inf_i \inf_j c(M_{j,i})\geq \inf_i \mathfrak{c}(K_i).\qedhere\]
\end{proof}

\begin{prop}\label{omitedges}
If a graph $G'$ arises from a graph $G$ by ommiting some edges, then the Cheeger constants satisfy: \[\mathfrak{c}(G')\leq \mathfrak{c}(G).\]
\end{prop}
\begin{proof}
For every finite set $M$ clearly the number $|\partial' M|$ is at most $|\partial M|$, where $\partial' M$ denotes the (vertex) boundary of $M$ in $G'$. Hence \[\mathfrak{c}(G')=\inf_{\textnormal{finite } M\subset V} \frac{|\partial' M|}{|M|}\leq \inf_{\textnormal{finite } M\subset V} \frac{|\partial M|}{|M|}=\mathfrak{c}(G).\qedhere\]
\end{proof}

This finishes our observations on the Cheeger constant of general graphs. We continue with the computation of the Cheeger constant of an infinite $2k$-regular tree, the Cayley graph of the free group of rank $k$.

\begin{prop}\label{CayleyFk}
The infinite $2k$-regular tree has Cheeger constant $\frac{2k-2}{2k-1}$.
\end{prop}

\begin{proof}
We want to find a lower bound for $c(M)$ for finite vertex sets $M$. By \autoref{MMR} we can assume that $M$ is connected. If the interior $\mathring{M}$ is empty, $c(M)=1$. So let us assume $|\mathring{M}| = n\geq 1$ and let $v$ be an inner vertex. We mark $v$ as the root of $M$ and divide the other vertices of $M$ in levels corresponding to their distance to $v$. Now we can view the tree $M$ from the root and estimate the total number of vertices: these are $v$ and its $2k$ neighbors in $M$ and every further inner vertex contributes $2k-1$ neigbors on a higher level that are not already counted. Hence $|M|\geq 1+2k+(n-1)(2k-1)=n(2k-1+\frac{2}{n})$ and 
\begin{align} c(M)=1-\frac{|\mathring{M}|}{|M|}\geq 1-\frac{n}{n(2k-1+\frac{2}{n})}=\frac{2k-2+\frac{2}{n}}{2k-1+\frac{2}{n}}.\label{freegroupineq}\end{align}
For increasing $n$ this expression is decreasing, whence we get the lower bound for $n\rightarrow \infty$: \[c(M)\geq \frac{2k-2}{2k-1}.\]
Choosing $M=B_n$ to be a ball with arbitrary center (and radius $n$), the inequality (\ref{freegroupineq}) becomes an equality, the sequence $B_n$ converges to the last expression and we obtain $\frac{2k-2}{2k-1}$ as the Cheeger constant of the infinite $2k$-regular tree.
\end{proof}

\subsubsection{Schreier Graphs}
\begin{defi}\label{Schreier graph}
To a group $\Gamma$, a subgroup $\Pi<\Gamma$ and a nonempty finite subset $S\subset \Gamma$ we assign the following directed graph $G_{\Gamma,\Pi,S}$: 
\begin{itemize}
\item The vertex set $V$ is the set of right cosets of $\Pi$ in $\Gamma$: \[V(G_{\Gamma,\Pi,S}):=\{\Pi\gamma|\gamma\in \Gamma\}.\] 
\item The edge set is \[E^+:=V\times S.\] The edge $(\Pi\gamma,s)$ has $\alpha((\Pi\gamma,s))\coloneqq \Pi\gamma$ as \emph{starting vertex} and $\beta((\Pi\gamma,s))\coloneqq\Pi\gamma s$ as \emph{terminal vertex}.
\item Moreover, we add the label $s$ to every edge $(\Pi\gamma,s)$.
\end{itemize}
If $S^\pm=S\sqcup S^{-1}$ is a generating set for $\Gamma$, the graph $G_{\Gamma,\Pi,S}$ is called the \emph{Schreier graph} of $\Gamma$ with respect to $\Pi$ and $S$.
\end{defi}

\begin{rem}\label{rem:vertorbit} We are particularly interested in the following setting: The group $\Gamma$ is the Veech group of an $L$-shaped Veech surface $L_{D,\epsilon}$ and acts on the points of $L_{D,\epsilon}$ (from the left). The Schreier graph we investigate is $G=G_{\Gamma,\Gamma_P,S}$ where $\Gamma_P$ is the stabilizer of a non-periodic connection point $P$ and $S$ is a finite generating set of $\Gamma$. In this setting, the vertices of the Schreier graph, the right cosets, can be identified with the points of the $\Gamma$-orbit of $P$:
\begin{itemize}
\item the subgroup $\Gamma_P$ itself is identified with $P$.
\item the vertex $\Gamma_P\gamma^{-1}$ is identified with $\gamma(P)=:\gamma\circ P$
\end{itemize}
This is well-defined because $\Gamma_P\gamma_1^{-1}=\Gamma_P\gamma_2^{-1}$ is equivalent to $\gamma_1^{-1}\gamma_2\in \Gamma_P$ and thus to $\gamma_1\circ P= \gamma_2\circ P$. Because of the left-right twist the label $s$ of an edge from $\Gamma_P\gamma^{-1}$ to $\Gamma_P\gamma^{-1}s$ becomes an $s^{-1}$ on the edge from $\gamma \circ P$ to $s^{-1}\gamma \circ P$. But for the question of amenability the direction of the edges does not matter, so we do not have to be too worried about.
\end{rem}

The definition of Schreier graphs looks very similar to the well-known concept of Cayley graphs and indeed if $\Pi$ is a normal subgroup of $\Gamma$, the Schreier graph is the same as the Cayley graph of the factor group $\quotient{\Gamma}{\Pi}$. But if $\Pi$ is not a normal subgroup, there are some important differences, e.g. the Schreier graph, in contrast to Cayley graphs, is not vertex-transitive as one can see in the following example:

\begin{exa}\label{F2modZ}
Let $\Gamma=\left<a,b | - \right>$ be the free group of rank $2$. We set $\Pi=\left<a\right>$, the subgroup generated by $a$, and $S=\{a,b\}$. The resulting graph $G_{\Gamma,\Pi,S}$ is presented in figure \ref{RL4Tpicture}. In particular, the edge $(\Pi,a)$ is a loop and it is the only loop. Hence every automorphism of $G_{\Gamma,\Pi,S}$ has to fix the vertex $\Pi$, which implies that  $G_{\Gamma,\Pi,S}$ is not vertex-transitive.
\end{exa}

\begin{figure}[h!]
\begin{center}
\includegraphics[scale=1.1]{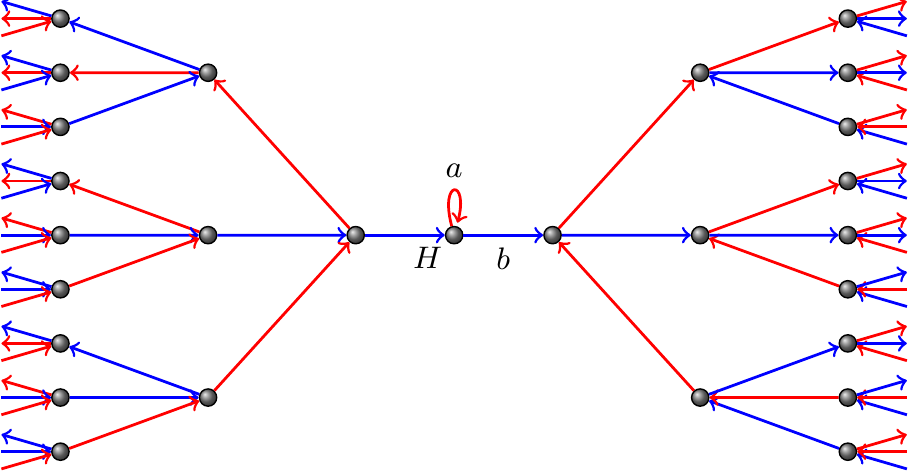}
\caption[Root-looped $4$-valent tree]{The \emph{root-looped $4$-valent tree}: a Schreier graph that is not vertex-transitive.}
\end{center}
\label{RL4Tpicture}
\end{figure}

\begin{defi}
We call the graph $G_{\Gamma,\Pi,S}$ from \autoref{F2modZ} above the \emph{root-looped $4$-valent tree} and denote it by $T_{rl4}$.
\end{defi}

\begin{prop}\label{cheegerconstRL4T}
The Cheeger constant of $T_{rl4}$ is $\mathfrak{c}(T_{rl4})=\frac{2}{3}$.
\end{prop}

\begin{proof}
First we observe that if we set $M\subset V$ as the root and its two neighbors, then $c(M)=\frac{2}{3}$, which implies $\mathfrak{c}(T_{rl4})\leq\frac{2}{3}$.

For the other direction let $M$ be any finite vertex set. If the root is not in $\mathring{M}$, we have seen in the proof of \autoref{CayleyFk} that $c(M)\geq \frac{2}{3}$. If the root is one of $n$ inner vertices of $M$, then using the same arguments as in the proof above $M$ has at least $1+2+3(n-1)$ vertices. Hence $c(M)\geq 1-\frac{n}{1+2+3(n-1)}=\frac{2}{3}$ and we get $\mathfrak{c}(T_{rl4})\geq\frac{2}{3}$, which finishes the proof.
\end{proof}

Later on, we want to choose the generating set $S$ such that it contains special elements. This will be allowed by a result of Woess about metrically equivalent graphs.

\begin{defi}
Let $G=(V,E)$ and $G'=(V',E')$ be graphs with edge metrics $d$ and $d'$, respectively. The graphs $G$ and $G'$ are called \emph{metrically equivalent}, if there exists a surjective map $\varphi: V\twoheadrightarrow V'$ and a constant $A\geq 1$ such that \[\frac{d(v,w)}{A}\leq d'(\varphi(v),\varphi(w))\leq A d(v,w)\] for all $v,w\in V$.
\end{defi}

\begin{prop}
Let $G=G_{\Gamma, \Pi, S}$ and $G'=G'_{\Gamma, \Pi, S'}$ be Cayley or Schreier graphs with respect to finite generating sets $S$ and $S'$ of $\Gamma$ and let $d$ and $d'$ be the edge metrics of $G$ and $G'$, respectively. Then $G$ and $G'$ are metrically equivalent.
\end{prop}

\begin{proof}
Since $S=\{s_1,\ldots, s_n\}$ and $S'=\{s_1',\ldots s_m'\}$ are finite and $G$ and $G'$ are connected $\max_{i=1,\ldots n; j=1, \ldots m}(d'(\Pi,\Pi s_i),d(\Pi,\Pi s_j'))$ is also finite. Set $A$ to be this number and let $\varphi$ be the identity map on $V(G)=V(G')$.
\end{proof}

\begin{prop}[\cite{Woe}, Theorem 4.7]\label{IndependFromGS}
Let $G$ and $G'$ be connected graphs with bounded vertex degrees. If $G$ and $G'$ are metrically equivalent, then $G$ is amenable if and only if $G'$ is amenable. In particular, for Cayley graphs and Schreier graphs amenability is independent of the choice of a finite generating set.
\end{prop}

\subsubsection{Free Subgroups}
Keeping in mind the equivalence between amenability of groups (cf.\@ e.g.\@ \cite{Amen}) and amenability of Cayley graphs (Proposition 12.4 in \cite{Woe}), as well as the similarities of Cayley graphs and Schreier graphs one could hope that the subgroup criterion for non-amenability of a group survives to Schreier graphs. This states that a group is non-amenable if it contains a non-amenable -- e.g. a free non-abelian -- subgroup. Translated to Schreier graphs where the subgroup $\Pi\triangleleft\Gamma$ is a normal subgroup, this means that if there exists a free non-abelian subgroup $\digamma<\Gamma$ having trivial intersection with $\Pi$, the Schreier graph -- which in this case is indeed the Cayley graph of the factor group -- is non-amenable. But unfortunately this translation is no longer true, if we drop the condition that $\Pi$ is normal in $\Gamma$, as we see in the following counterexample.

\begin{exa}
Let $\Gamma=\left<a,b \,|\, - \right>$ be the free group in two generators and $\Pi$ the subgroup $\left<\{a^k b a^{-k}\}_{k\in \mathbb{Z} - \{0\}}\right>$. Then the Schreier graph $G_{\Gamma,\Pi,\{a,b\}}$ is amenable, but there is a free non-abelian subgroup $\digamma$ having trivial intersection with $\Pi$.
\end{exa}

\begin{proof}
All cosets $\Pi a^k$ are different, because \[\Pi a^m=\Pi a^n \Leftrightarrow \Pi a^{m-n}=\Pi \Leftrightarrow a^{m-n} \in \Pi\Leftrightarrow m=n.\] From $\Pi a^k$ there is an outgoing edge labeled with $a$ to $\Pi a^{k+1}$ and an incoming edge from $\Pi a^{k-1}$. Since $\Pi a^kb=\Pi a^k \Leftrightarrow a^kba^{-k}\in \Pi$ the $b$-edges from $\Pi a^k$ are loops at $\Pi a^k$ for all $k\in \mathbb{Z} - \{0\}$. Summing up, this part of the coset graph looks like the Cayley graph of $\mathbb{Z}$ with an additional loop at almost every vertex. Only at the origin, the vertex $\Pi=\Pi 1$, there are an incoming and an outgoing $b$-edge, connecting this part to the following part:

All cosets $\Pi w$ and $\Pi w'$, with $w$ and $w'$ different reduced words starting with $b^{\pm 1}$, are different vertices since $\Pi w=\Pi w' \Leftrightarrow ww'^{-1}\in \Pi$, and this is not true since every (non-trivial) element in the subgroup $\Pi$ begins and ends with $a$ or $a^{-1}$. Furthermore, no $\Pi w$ for $w$ starting with $b^{\pm 1}$ is adjacent to one of the $\Pi a^k$ ($k\neq 0$), because of the $4$-regularity of the Schreier graph. So the Schreier graph may be pictured as in \autoref{counterexamplePicture}.

\begin{figure}[h]\label{counterexamplePicture}
\begin{center}
\includegraphics[scale=0.8]{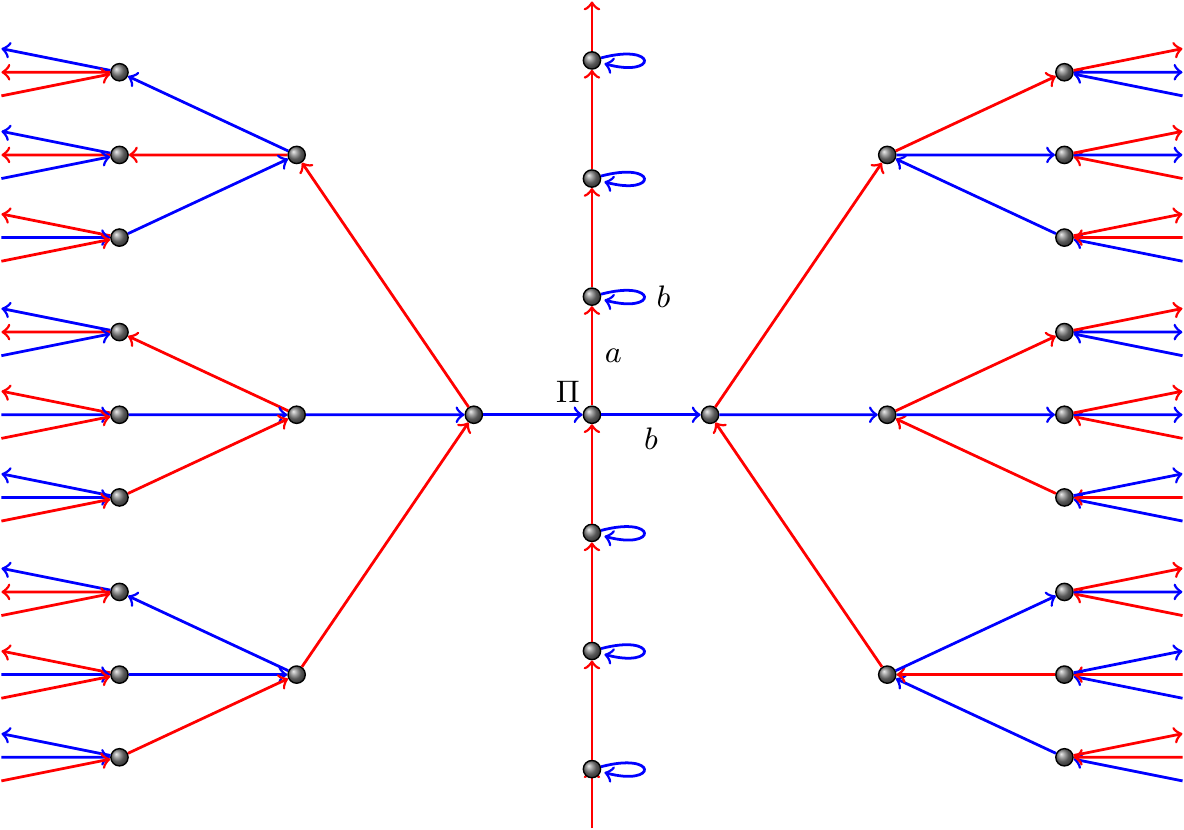}
\end{center}
\caption{Schreier graph $G_{\Gamma,\Pi,S}$: amenable although $\Gamma$ contains a free subgroup intersecting $\Pi$ only trivially.}
\end{figure}

Now we can give a sequence of finite vertex sets with the property that the quotient of boundary vertices to all vertices of one set tends to zero: \[(M_i)_{i\in \mathbb{N}} \textnormal{ with } M_i=\{\Pi a^k \colon |k|\leq i \}.\]
In $M_i$ there are three boundary vertices $\Pi$, $\Pi a^i$ and $\Pi a^{-i}$. On the other hand, we have $|M_i|=2i+1$, so $c(M)=\frac{3}{2i+1} \rightarrow 0$ for $i\rightarrow \infty$.

It remains to find a non-abelian free subgroup $\digamma<\Gamma$ that intersects with $\Pi$ only trivially. One can easily check that the subgroup $\digamma=\left<bab,b^2\right>$ is such a group: As a subgroup of a free group it is free, obviously of rank $>1$, and every nontrivial word in $\digamma$ begins with $b^{\pm 1}$ and thus is not in $\Pi$.
\end{proof}

\subsection{The Combinatorial Laplacian and the Cheeger Constant}
In \autoref{critexp} we will need some information on the spectrum of a graph $G$ with infinite vertex set, in particular on its bottom $\mu_0(G)$. This is why we provide some basic facts about the combinatorial Laplacian in this section.

Let $G=(V,E)$ be a connected graph with maximal valency $k>0$. 
We define the \emph{gradient} $\nabla$ as an operator sending a map $b:V\rightarrow \RR$ to the map \[\nabla b: E\rightarrow \RR, \quad  (i,j)\mapsto b(i)-b(j)\] and the \emph{combinatorial Laplacian} $\LapC$ as \[\LapC b: V\rightarrow \RR, \quad i\mapsto \sum_{j : \{i,j\}\in E}b(i)-b(j)\] for all maps $b$, which are in $\mathfrak{l}^2(V)$, i.e.\ the maps with $\langle b,b\rangle\coloneqq \sum_i b(i)^2<\infty$. The inner product of $a,b: V\rightarrow \RR$ is defined as $\langle a,b \rangle\coloneqq \sum_{i\in V}a(i)b(i)$. For functions $c,d:E\rightarrow \RR$ we use the same notation of inner product (defined as $\langle c,d \rangle\coloneqq \sum_{e\in E}c(e)d(e)$) and norm $||c||$. The combinatorial Laplacian is a bounded operator self-adjoint with respect to the quadratic form $q(b)=\sum_{\{i,j\}\in E}(b(i)-b(j))^2=\langle \LapC b,b\rangle$.
If $G$ has a finite vertex set $V$, the constant maps are eigenfunctions with eigenvalue $\lambda_0(G)=0$. For connected graphs the eigenvalue $0$ has multiplicity $1$ and the difference $\mu_0(G)=\lambda_1(G)-\lambda_0(G)$ is called \emph{spectral gap}. If $V$ is infinite, we set the \emph{spectral gap} $\mu_0(G)$ to be the smallest eigenvalue $\mu_0(G)$ of $\LapC$.

Let us now assume that $G$ has infinitely many vertices. By the Min-Max principle the spectral gap satisfies \begin{equation}
\mu_0(G)=\inf_b\frac{\sum_{\{i,j\}\in E}(b(i)-b(j))^2}{\sum_{i\in V}b(i)^2}=\inf_b\frac{||\nabla b||^2}{||b||^2}.\label{muNull}
\end{equation}

There are upper and lower bounds on $\mu_0(G)$ in terms of the Cheeger constant $\mathfrak{c}(G)$:
\begin{prop}[Cheeger]\label{prop: CheegerSpecGap}
Let $G=(V,E)$ be a connected graph with $|V|=\infty$ and maximal vertex valency $k$. Then the following inequalities for the spectral gap $\mu_0(G)$ hold: \[\frac{\mathfrak{c}(G)^2}{2k}\leq \mu_0(G) \leq k \mathfrak{c}(G).\]
\end{prop}
A proof of a similar theorem can be found in \cite{CdV}, Section $2$. But note that he defines the Cheeger constant via the edge boundary $\partial_E A$, which consists of all edges that connect vertices of $A$ and $A^c$. We will just use the first inequality and thus also prove just this one:

\begin{proof}
Let $b:V\rightarrow \RR$ have finite support. We define \[\mathcal{S}\coloneqq \sum_{\{i,j\}\in E}|b^2(i)-b^2(j)|\] and by the Cauchy-Schwartz inequality for $\langle \nabla b, b(i)+b(j) \rangle$ combined with the estimate $\sum_{\{i,j\}\in E}|b(i)|+|b(j)|\leq 2k \sum_{i\in V}|b(i)|$ we obtain \[\mathcal{S}\leq \sqrt{2k}||\nabla b|| \cdot ||b||.\]

On the other hand $\mathcal{S}=\sum b^2(i)-b^2(j)$, where the sum is taken over all (oriented) edges $(i,j)$ with $b^2(i)\geq b^2(j)$. The image of $b$ is finite and we sort the values of $b^2$ obtaining $a_0=0<a_1<\ldots < a_r$ and define \[A_l\coloneqq \{i\in V \mid b^2(i)\geq a_l\}.\]
Then we can write the sum $\mathcal{S}$ as $\sum (a_l-a_{l-1})$, where each summand $a_l-a_{l-1}$ appears with multiplicity equal to the number of edges $(i,j)$ with $b^2(i)\geq a_l$ and $b^2(j)<a_l$. This number is bounded below by $|\partial A_{l}|$. With the definition of the Cheeger constant $\mathfrak{c}(G)<\frac{|\partial A_{l}|}{|A_{l}|}$ and thus \[\mathcal{S}\geq \mathfrak{c}(G)\sum_{l=1}^r(a_l-a_{l-1})|A_l|=\mathfrak{c}(G)||b||^2\] follows. All together we finish the proof by obtaining \[\frac{||\nabla b||^2}{||b||^2}\geq \frac{\mathfrak{c}(G)^2}{2k}.\] 
\end{proof} 

Thus for $G$ having a strictly positive spectral gap $\mu_0(G)$ is equivalent to having a strictly positive Cheeger constant $\mathfrak{c}(G)$ and hence also to being non-amenable.


\section{Critical Exponent and Graph-Periodic Manifolds}\label{critexp}
In \autoref{Sizes} we already saw Veech groups of different sizes. Another way to measure the size of Fuchsian groups is the \emph{critical exponent} which is introduced in the first part of this section. The aim of the second part is to build a bridge from specific Fuchsian groups -- subgroups of lattices -- and their critical exponent to amenability of Schreier graphs (described in \autoref{sect:Schreier}). We use a concept introduced by Tapie (\cite{TapFr}): \emph{graph-periodic manifolds} $M$ over a cell $C$, i.e. manifolds consisting of isometric copies of another manifold glued together according to the structure given by a graph. Following \cite{RobTap} we compare the bottom of the spectrum of $C$ and of $M=\quotient{\HH}{\Pi}$. In doing so we prove that the critical exponent of $\Pi<\Gamma$ is strictly smaller than $1$, if the Schreier graph $G_{\Gamma,\Pi,S}$ is non-amenable for a finite set $S$ generating $\Gamma$. 

Further background on the critical exponent viewed from different perspectives can be found in \cite{Nic}.

\subsection{The Critical Exponent}
First we define the term critical exponent and collect some basic properties. Let $\rho_\HH$ be the hyperbolic metric on the upper half plane $\HH$.

\begin{defi}
Let $\Gamma$ be a Fuchsian group and $\ast\in \HH$. The \emph{Poincaré series} to the exponent $a\in\RR$ and the base point $\ast$ is the series $\sum_{\gamma\in \Gamma} e^{-a\rho_\HH(\ast,\gamma(\ast))}$. The infimum of exponents $a$, for which the Poincaré series converges is called the \emph{critical exponent} $\delta(\Gamma)$: \[\delta(\Gamma)\coloneqq \inf \bigl\{a\in \mathbb{R} \mid \sum_{\gamma\in \Gamma} e^{-a\rho_\HH(\ast,\gamma(\ast))}<\infty\bigr\}\]
\end{defi}

Because of the triangle inequality the convergence of the Poincaré series and thus also the critical exponent are independent of the choice of base point. Usually we will set $\ast=i$. Our next observation concerns the critical exponent of commensurable groups. Recall that two subgroups $\Gamma$ and $\Gamma'$ of $\SL_2(\RR)$ are called commensurable if there exist subgroups $\Pi<\Gamma$ and $\Pi'<\Gamma'$ of finite index each, that are conjugate in $\SL_2(\RR)$.

\begin{prop}\label{prop: commenssameCE}
Commensurable Fuchsian groups have the same critical exponent.
\end{prop}

\begin{proof}
We prove this proposition in two steps. First, conjugate Fuchsian groups have the same critical exponent since conjugation just corresponds to a change of the base point in the Poincaré series. Second, descending to a subgroup $\Pi$ of finite index in a group $\Gamma$ also does not change the critical exponent: the inequality $\delta(\Pi)\leq \delta(\Gamma)$ is clear by definition. The reverse inequality is obtained by a rearrangement of the summands of the Poincaré series.
\end{proof}

Obviously the critical exponent of infinite groups is at least $0$. By direct computations one can check some further general bounds on the critical exponent of Fuchsian groups $\Gamma$: cyclic groups generated by a hyperbolic element have critical exponent $0$, cyclic parabolic groups have critical exponent $\frac{1}{2}$. Since the critical exponent of subgroups is less or equal to the critical exponent of the big group, every group containing a parabolic element has critical exponent at least $\frac{1}{2}$.

For non-elementary groups Beardon and later Paterson proved a better bound:
\begin{prop}[\cite{Bea}/\cite{PatExp}]\label{thm:BeaPat}
If $\Gamma$ is non-elementary and contains a parabolic element, then $\delta(\Gamma)> \frac{1}{2}$.
\end{prop}

We cite two more theorems about critical exponents of Fuchsian groups:

\begin{prop}[\cite{Nic}, Theorem 1.6.1]\label{thm:Ni1}
For all Fuchsian groups $\Gamma$ the critical exponent is at most $1$.
\end{prop}

\begin{prop}[\cite{Nic}, Theorem 1.6.3]\label{thm:Ni2}
If $\Gamma$ is a lattice, then $\delta(\Gamma)=1$.
\end{prop}

Another very useful result concerns a connection between the critical exponent $\delta(\Pi)$ of a Fuchisan group $\Pi$ and  the smallest eigenvalue $\lambda_0(M)$ of the Laplacian $\Delta_M$ on the hyperbolic manifold $M=\quotient{\HH}{\Gamma}$. It was first proven by Patterson in \cite{Pat} for geometrically finite Fuchsian groups and generalized by Sullivan in \cite{Sul} to all discrete groups acting on the hyperbolic space $\HH^{d+1}$ by isometries\footnote{Note that $\HH=\HH^2=\HH^{1+1}$.}:

\begin{prop}[\cite{Pat}/\cite{Sul}] Let $\Gamma$ be a Fuchsian group and $M=\quotient{\HH}{\Gamma}$. Then the smallest eigenvalue $\lambda_0$ of the Laplacian on $M$ is \[\lambda_0(M)=\begin{cases} \delta(\Gamma)(1-\delta(\Gamma)), & \textnormal{if } \delta(\Gamma)\geq \frac{1}{2} \\ \frac{1}{4}, & \textnormal{if } \delta(\Gamma)\leq \frac{1}{2}. \end{cases}\]\label{prop: PattersonSullivan}
\end{prop}

In particular, $\lambda_0(M)>0$ implies $\delta(\Gamma)<1$. This is why we will take a closer look at the Laplacian on hyperbolic manifolds and the bottom of its spectrum in the following.	

\subsection{From Hyperbolic Surfaces to Schreier Graphs}
In this section we follow the ideas of Tapie (\cite{TapFr}) and Roblin (\cite{RobTap}). They discuss Riemannian coverings and a generalization -- graph-periodic manifolds -- in order to obtain a lower bound on the smallest eigenvalue of the Laplacian on the manifolds. Roughly spoken, a graph-periodic manifold consists of isometric copies of a smaller manifold with boundary, which are glued together according to the structure given by a regular graph. The lower bound is given in Théorème 0.2 of \cite{RobTap} for Riemannian coverings and in Théorème 1 of \cite{TapFr} for graph-periodic manifolds. We will introduce the notion of graph-periodic manifolds and explain how our situation fits into this concept. The graph-periodic manifold is $\quotient{\HH}{\Pi}$ and the smaller manifold is $\quotient{\HH}{\Gamma}$. In the above mentioned theorems there is the assumption that $\Pi$ is a normal subgroup of $\Gamma$. We summarize the proofs to check that the statements still hold without
supposing $\Pi < \Gamma$ being a normal subgroup.

\begin{defi}
Let $k\in\NN$ be fixed. A \emph{marked cell of valency $k$} is a set $\{C,H_1,\ldots,H_k\}$ where $C$ is a smooth Riemannian manifold with piecewise $\mathcal{C}^1$ boundary and $H_i\subset \partial C$ are pairwise disjoint compact codimension $1$ submanifolds which are $\mathcal{C}^1$ with piecewise $\mathcal{C}^1$ boundary (of codimension $2$). The $H_i$ are called \emph{transition zones}, $C$ is called the \emph{cell}.
\end{defi}
\begin{defi}
Let $G=(V,E)$ be a graph of constant valency $k$ and $\{C,H_1,\ldots,H_k\}$ be a marked cell of valency $k$. A manifold $M$ is called \emph{marked $G$-periodic over $C$} if it satisfies the following conditions:
\begin{enumerate}
\item For all $v\in V$ there exist submanifolds $C_v\subset M$ with pairwise disjoint interiors such that $M=\bigcup_{v\in V}C_v$. Furthermore there exists an isometry $J_v:C_v\rightarrow C$.
\item For all $v,w\in V$ there is an edge $\{v,w\}\in E$ if and only if $J_v(C_v\cap C_w)\subset \partial C$ contains a unique transition zone, which we will denote by $H_{vw}$.
\item For all edges $\{v,w\}\in E$ the map $J_w\circ J_v^{-1}:J_v(C_v\cap C_w)\rightarrow J_w(C_v\cap C_w)$ induces an isometry from $H_{vw}$ to $H_{wv}$.
\end{enumerate}
\end{defi}

\paragraph{Subgroups of Fuchsian Groups.}
Let us now describe how the situation of this article fits into the concept of graph-periodic manifolds. Afterwards we will look at a small example using well-known groups.

Given a lattice Fuchsian group $\Gamma$ and a subgroup $\Pi$ there are the corresponding quotients $M=\quotient{\HH}{\Gamma}$ and $N=\quotient{\HH}{\Pi}$. Thus there are three covering maps: $p:N\rightarrow M$ and the universal coverings $\pi_{\Gamma}:\HH\rightarrow M$ and $\pi_\Pi:\HH\rightarrow N$. Since $\Gamma$ is a lattice, there is a fundamental domain $\mathcal{F}_\Gamma\subset \HH$ of finite volume and with finitely many sides for the action of $\Gamma$ on $\HH$, such that pairs of sides correspond to a set $S=\{\gamma_1,\ldots,\gamma_n\}$ of elements generating $\Gamma$ (usually a Dirichlet fundamental domain). We set $C=\overline{\mathcal{F}_\Gamma}$ and mark $C$ by choosing for any of these pairs a transition zone $H_i$ on one of the sides and the $\gamma_i$-image of $H_i$ on the other side. Then $C$ is a marked cell of valency $2\cdot n$ and $N$ consists of $\Gamma$-images of $\pi_\Pi(C)$, one for each coset $\Pi\gamma$. Setting $G$ to be the Schreier graph $G_{\Gamma,\Pi,S}$, we see that $N$ is a $G$-periodic manifold over $C$.

\begin{exa}
Let $\Gamma$ be the modular group $\SL_2(\ZZ)$ and $\Pi$ be the principal congruence subgroup $\Gamma[2]$. The fundamental domains $\mathcal{F}_\Gamma$, $\mathcal{F}_{\Gamma[2]}$ and the Schreier coset graph $G_{\Gamma,\Gamma[2],\{S,T\}}$ for $S=\smatr{0}{-1}{1}{0}$ and $T=\smatr{1}{1}{0}{1}$ are shown in \autoref{pic:SL2ZGPeriodic}. The main difference between this example and the situation we are interested in is that in this example the subgroup has finite index; hence the Schreier graph $G$ is finite and thus in particular amenable. Moreover, $M=\quotient{\HH}{\Gamma}$ and $N=\quotient{\HH}{\Gamma[2]}$ have finite volume and thus the smallest eigenvalues $\lambda_0(M)$ and $\lambda_0(N)$ are $0$, since any constant function $\phi_0:M$ or $N \rightarrow \RR$ is a valid eigenfunction with eigenvalue $0$.
\end{exa}

\begin{figure}[ht]
\begin{center}
\includegraphics[width=0.9\linewidth]{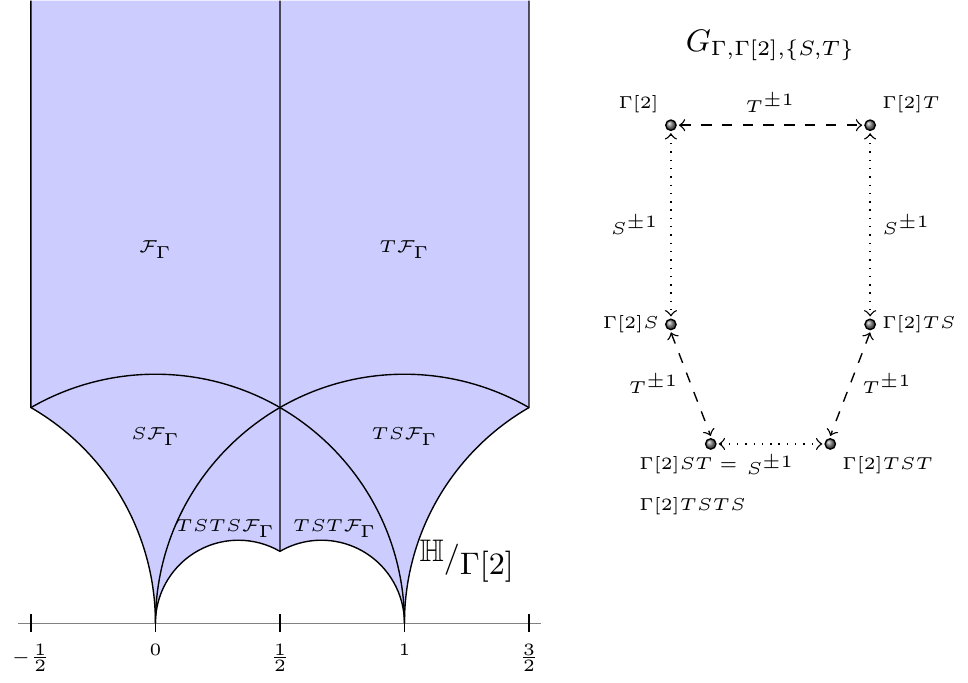}
\caption{$\quotient{\HH}{\Gamma[2]}$ as $G$-periodic manifold over a fundamental domain $\mathcal{F}_\Gamma$ of $\Gamma=\SL_2(\ZZ)$.}
\label{pic:SL2ZGPeriodic}
\end{center}
\end{figure}

Before we state the main result of this section we collect some facts about the Laplacian $\Delta_M$ on a hyperbolic manifold $M=\quotient{\HH}{\Gamma}$. 
For the basic definitions of divergence $\div f$, gradient $\nabla f$ and the Laplacian $\Delta=\div(\nabla f)$ see Chapter $1$ of \cite{Cha}.

For a hyperbolic manifold $M$ let $L^2(M)$ be the space of measurable maps \mbox{$f:M\rightarrow \RR$} with $\int_M |f|^2 < \infty$. On $L^2(M)$ there is the \emph{inner product} $\langle f,g\rangle_M\coloneqq \int_M fg$ and the \emph{norm} $||f||^2_M\coloneqq \langle f,f\rangle_M$.
Furthermore, we define the \emph{Rayleigh quotient} as \[\mathcal{R}_M(f)=\frac{||\nabla f||^2_M}{||f||^2_M}.\] The Min-Max-principle implies that the smallest eigenvalue $\lambda_0(M)$ satisfies \begin{equation}\lambda_0(M)=\inf_{f\in \mathcal{H}^1(M)} \mathcal{R}_M(f),\label{eqn: Rayleigh}\end{equation} where $\mathcal{H}^1(M)\subset L^2(M)$ is the Sobolev space of maps $f\in L^2(M)$, whose gradient is an $L^2$ vector field.
Moreover, if $\phi_0:M\rightarrow \RR$ is a eigenfunction to the eigenvalue $\lambda_0(M)$, then 
\[\lambda_1(M)=\inf \{\mathcal{R}_M(f) \mid f\in \mathcal{H}^1(M), \langle f,\phi_0\rangle_M=0\}.\]

The lower bound on $\lambda_0(\quotient{\HH}{\Pi})$ obtained in the main result will depend on the spectral gap $\nu\coloneqq \lambda_1(C)-\lambda_0(C)$ of the Laplacian on $C$ (with Neumann boundary conditions). Therefore it is important that in our case, where $C$ is a (Dirichlet) fundamental domain of a lattice $\Gamma$, the smallest eigenvalue $\lambda_0(C)$ is isolated and thus the spectral gap is strictly positive. This is shown in three steps:

\begin{itemize}
\item By Lemme 5.2 of \cite{RobTap} the eigenvalue $\lambda_0(C)$ is an isolated eigenvalue of multiplicity $1$, if $\lambda_0(C)<\lambda_0^{ess}(C)$, where $\lambda_0^{ess}(C)$ is the bottom of the essential spectrum of $C$, i.e.\ the infimum of the real numbers for which a so-called Weyl's sequence exist.
\item Since $C$ has finite volume, the constant functions are in $\mathcal{H}^1(C)$ and thus $\lambda_0(C)=0$.
\item It holds $\lambda_0^{ess}(C)>0$:
\end{itemize}

\begin{lem}[\cite{RobTap}, Lemme 5.4] Let $C$ be a Dirichlet fundamental domain of a lattice. The bottom of the essential spectrum $\lambda_0^{ess}(C)$ is at least $\frac{1}{4}$.
\end{lem}

\begin{proof}
By Lemme 5.3 of \cite{RobTap}, if there exists a compact subset $K\subset C$ and a function $\phi: C-K \rightarrow (0,\infty)$ with gradient $\nabla \phi$ tangent on $\partial C - K$ that satisfies $\Delta\phi \geq \lambda \phi$, then $\lambda_0^{ess}(C)\geq \lambda\in \RR$.

Since $C$ is the Dirichlet fundamental domain of a lattice, there exists a compact set $K$ such that $C-K$ is a disjoint union of finitely many cusps. We define such a function $\phi$ only at a cusp at $\infty$, since all cusps are conjugate to such a cusp. It has the form $\{(x,y)\in\HH \mid a\leq x \leq b\ \land y\geq c\}$ for some $a,b,c\in\RR$ with $c>0$. Setting $\phi((x,y))=y^{\frac{1}{2}}$, one easily sees that $\phi$ is strictly positive, $\nabla\phi$ is tangent on $\partial C-K$ and that $\Delta \phi = \frac{1}{4} \phi$. Hence $\lambda_0^{ess}(C)\geq \frac{1}{4}$.
\end{proof}

Summarized this yields:
\begin{prop}\label{prop:etagr0}
Let $C$ be a Dirichlet fundamental domain of a lattice. The spectral gap $\nu\coloneqq \lambda_1(C)-\lambda_0(C)=\lambda_1(C)$ of the Laplacian on $C$ (with Neumann boundary conditions) is strictly positive: $\nu>0.$
\end{prop}

Let us now state the main result of this section. The proof is basically a special case of the proof of Théorème 4.3 of \cite{RobTap}.

\begin{prop}\label{prop: RT}
Let $\Gamma$ be a lattice, $\Pi<\Gamma$ a subgroup, $C$ a Dirichlet fundamental domain of $\Gamma$ and $G=G_{\Gamma,\Pi,S}$ the Schreier graph of $\Gamma$ with respect to $\Pi$ and the finite generating set $S$ obtained from $C$. Then $N=\quotient{\HH}{\Pi}$ is a marked $G$-periodic manifold over $C$ as described above. The bottom of the spectrum $\lambda_0(N)$ can be bounded above by
\[\lambda_0(N)\geq \min\left\{\frac{\frac{A}{V}\nu}{\nu+\frac{A}{V}\mu_0(G)}\mu_0(G),\nu\right\},\] where $A>0$ is a constant depending on the geometry of $C$ in a neighborhood of the transition zones, $\nu>0$ is the spectral gap of $C$, $V$ is the (finite) volume of $C$ and $\mu_0(G)$ is the spectral gap of the combinatorial Laplacian on the graph $G$.
\end{prop}

Combining this proposition with \autoref{prop: PattersonSullivan}, \autoref{prop: CheegerSpecGap} and \autoref{IndependFromGS} this implies:
\begin{cor}\label{prop: Tapie}
Let $\Gamma$ be a lattice and $\Pi<\Gamma$ a subgroup. If the Schreier graph $G_{\Gamma,\Pi,S}$ is non-amenable for any finite set $S$ generating $\Gamma$, then the critical exponent $\delta(\Pi)$ is strictly smaller than $1$.
\end{cor}

\begin{proof}[Proof of \autoref{prop: RT}.]
First observe that since $\Gamma$ is a lattice, $C$ has finite volume $V$ and thus the constant map $\EE_C: C\rightarrow \{1\}\subset\RR$ is an eigenfunction for the eigenvalue $\lambda_0(C)=0$. We can lift $\EE_C$ to the map $\EE_N:N\rightarrow\{1\}$, but in the following we will just write $\EE$ for both maps.

In \eqref{eqn: Rayleigh} we saw that $\lambda_0(N)=\inf_{f\in \mathcal{H}^1(N)}\mathcal{R}(f)$. Let $f_\epsilon$ be a smooth map with compact support satisfying $\mathcal{R}(f_\epsilon)=\frac{||\nabla f_\epsilon||_N}{||f_\epsilon||_N}\leq \lambda_0(N)+\epsilon$.

For every vertex $i\in V(G)$ we write $C_i$ for the image of the cell $C$ corresponding to this vertex and define
\begin{itemize}
\item $f_{\epsilon,i}\coloneqq f_\epsilon|_{C_i}$,
\item $b_i\coloneqq \frac{1}{V}\langle f_{\epsilon,i},\EE\rangle_{C_i}=\frac{1}{V}\int_{C_i} f_{\epsilon,i}\cdot \EE$, and
\item $g_i\coloneqq f_{\epsilon,i}-b_i\EE$.
\end{itemize} 

One easily checks that $g_i$ and $\EE$ are orthogonal (with respect to the inner product on $C_i$) and hence $\mathcal{R}(g_i)\geq \lambda_1(C)=\nu$. Moreover by bilinearity of the inner product we have $||g_i||_{C}^2=||f_{\epsilon,i}||_C^2-b_i^2V$.

We want to estimate $\lambda_0(N)+\epsilon$ and know that it is at least $\frac{||\nabla f_\epsilon||_{N}^2}{||f_\epsilon||_{N}^2}=\frac{\sum_{i\in V(G)}||\nabla f_\epsilon||_{C_i}^2}{\sum_{i\in V(G)}||f_\epsilon||_{C_i}^2}$. We continue by giving a lower bound on the numerator and an upper bound on the denominator.

The first is done in Lemme 4.8 of \cite{RobTap}, which states, that there is a constant $A>0$ depending on neighborhoods of the transition zones such that \[\sum_{i\in V(G)}||\nabla f_{\epsilon,i}||_{C_i}^2\geq A \sum_{\{i,j\}\in E(G)} (b_i-b_j)^2.\]

The proof is done by estimating $||\nabla f_{\epsilon,i}||_{H_{ij}\times [0,R]}^2$ in terms of the Newtonian Capacity of the ``rectangles'' $H_{ij}\times [-R,R]$. These are tubular neigborhoods of the transition zones, which have one part in the cell $C_i$ and one in the cell $C_j$ for each edge $\{i,j\}\in E(G)$. This is technical and we will not give the details here.

For the upper bound on $\sum_{i\in V(G)}||f_\epsilon||_{C_i}^2$ let us estimate $||\nabla f_{\epsilon,i}||_{C_i}^2=||\nabla (b_i\EE) +\nabla g_i||_{C_i}^2$. The first summand obviously vanishes and we obtain \begin{align*}
(\lambda_0(N)+\epsilon)\sum_{i\in V(G)}||f_{\epsilon,i}||_{C_i}^2 &\geq \sum_{i\in V(G)}||\nabla f_{\epsilon,i}||_{C_i}^2\\& \geq \sum_{i\in V(G)}||\nabla g_i||_{C_i}^2\\ & \geq\sum_{i\in V(G)}\nu ||g_i||_{C_i}^2\\ &\geq \nu \sum_{i\in V(G)}||\nabla f_{\epsilon,i}||_{C_i}^2 - \nu\sum_{i\in V(G)} b_i^2V.\end{align*} 
Assume that $\lambda_0(N)+\epsilon < \nu$. Otherwise the proposition is clearly true. Then the above inequality is equivalent to 
\[\sum_{i\in V(G)}||f_{\epsilon,i}||_{C_i}^2 \leq V\left(1-\frac{\lambda_0(N)+\epsilon}{\nu}\right)^{-1}\sum_{i\in V(G)} b_i^2 .\]

With these bounds on the numerator and denominator we summarize:
\[\lambda_0(N)+\epsilon\geq \frac{\sum_{i\in V(G)}||\nabla f_\epsilon||_{C_i}^2}{\sum_{i\in V(G)}||f_\epsilon||_{C_i}^2} \geq \frac{A}{V} \frac{\sum_{\{i,j\}\in E(G)}(b_i-b_j)^2}{\sum_{i\in V(G)} b_i^2}\left(1-\frac{\lambda_0(N)+\epsilon}{\nu}\right).\]

By \eqref{muNull} the quotient $\frac{\sum_{\{i,j\}\in E(G)}(b_i-b_j)^2}{\sum_{i\in V(G)} b_i^2}$ is greater or equal to the spectral gap $\mu_0$ of the graph $G$. Inserting this observation and solving for $\lambda_0(N)+\epsilon$ yields 
\[\lambda_0(N)+\epsilon\geq \frac{\frac{A}{V}\nu}{\nu+\frac{A}{V}\mu_0(G)}\mu_0(G).\] Since all constants on the right-hand side are independent of $\epsilon$ this finishes the proof.
\end{proof}


\section{Prototypes of Veech Surfaces in $\Omega M_2(2)$}
In this section we take a closer look at the prototypes $L_D$ defined in \autoref{subs:Lattices}. All results in this section in principle also hold for the other two types of surfaces $L_{D,\pm 1}$ from \autoref{def:LDorLDeps}. One just has to change some numbers. This is why we will not give detailed proofs for these types, but summarize the results at the end of this section.
\subsection{The Horizontal and Vertical Cylinder Decompositions of $L_D$}\label{subs:AandB}
There are two parabolic elements of $\SL(L_D)$ that can be found easily: one fixing the horizontal, one fixing the vertical direction. They can be found by looking at the horizontal respectively vertical cylinder decomposition of the surface $L_D$. On each cylinder the parabolic element will act as a (possibly multiple) Dehn twist. For the calculations we set $d\coloneqq\frac{D}{4}$ use the \emph{modulus} of a cylinder, which is defined as the quotient of circumference $c$ and height $h$ of the cylinder: $\mu\coloneqq \frac{c}{h}$. 

Let us start with the horizontal decomposition. The surface $L_D$ decomposes into two cylinders whose circumferences are in the horizontal direction: the upper cylinder $C_u=\quotient{[0,1]\times(1,w)}{\sim}$ and the lower cylinder $C_d=\quotient{[0,1+w]\times(0,1)}{\sim}$. The moduli are $\mu_u=\frac{1}{w-1}$ and $\mu_d=\frac{1+w}{1}$. The quotient of the moduli is $\frac{\mu_d}{\mu_u}=(w+1)(w-1)=w^2-1=\frac{d-1}{1}$. Hence the desired parabolic element in $\SL(L_D)$ is \[\underline{\underline{B\coloneqq\matr{1}{1+w}{0}{1}}}=\matr{1}{1\cdot \mu_d}{0}{1}=\matr{1}{(d-1)\cdot \mu_u}{0}{1}.\]
This means $B$ twists the lower cylinder once and the upper cylinder $d-1$ times. Thus we obtain for a point $\PointL{x}{y}\in C_d$: \[\matr{1}{1+w}{0}{1}\circ\PointL{x}{y}=\PointL{x+(1+w)y \mod (1+w)}{y}\]

For the action on points of $C_u$ we have to bear in mind that the lower left corner of the upper cylinder's closure has coordinates $\PointL{0}{1}$, not $\PointL{0}{0}$. This is why we have to replace $y$ by $y-1$, apply the action by the Dehn twist and then shift the image point $1$ upwards again:
\[\matr{1}{1+w}{0}{1}\circ\PointL{x}{y-1}=\PointL{x+(1+w)(y-1) \mod 1}{y-1+1}\]

Note that, as the $y$-coordinate is not changed by $B$, the image points always lie in the same cylinder as $\PointL{x}{y}$. Knowing this and the description of the action of $B$, we also know the action of $B^l$ for all $l\in\ZZ$:

\begin{equation}
B^l\circ \PointL{x}{y}=
\begin{cases}
\PointL{x+l(1+w)y \mod (1+w)}{y} & \text{ if } y\leq 1\\
\PointL{x+l(1+w)(y-1) \mod 1}{y} & \textnormal{ if } y>1.
\end{cases}\label{eqn:Bl}
\end{equation}

The points of $L_D$ we are interested in are the connection points which by \autoref{prop: ConnPoints} are the points with coordinates in $\QQ(w)$. We will write them in the form $\PointL{x}{y}=\PointL{x_r+x_iw}{y_r+y_iw}$ and call $x_r,y_r\in\QQ$ the \emph{\rp s} and $x_i,y_i\in\QQ$ the \emph{\ip s}.

For the horizontal parabolic element $B=\smatr{1}{1+w}{0}{1}$, a point $Q=\PointL{x_r+x_iw}{y_r+y_iw}$ and an integer $l$ the $x$-coordinate of the point $B^l\circ Q$ is denoted by $x_{B^l}$. For its \rp \ and \ip\ we write $x_{B^l,r}$ and $x_{B^l,i}$, respectively. The difference $x_{B^l,i}-x_i$ is denoted by $\Delta_{B^l}(Q)$.

For points $Q=\PointL{x}{y}=\PointL{x_r+x_iw}{y_r+y_iw}\in L_D$ with $y\leq 1$, i.e. $Q\in C_d$, \eqref{eqn:Bl} states $x_{B^{l}}=x+ l (1+\om)y \mod (1+\om)$. Hence the difference of the \ip s amounts to \[\Delta_{B^l}(Q)=l (y_r+y_i) - q_{y,l} \textnormal{ with } q_{y,l}=\lfloor ly \rfloor \textnormal{ or } \lceil ly \rceil,\]
which implies that for all $Q\in C_d$ and all $l\in\ZZ$ there exists an $r\in (-1,1)$ such that $\Delta_{B^l}(Q)=ly_r+ly_i -ly - r = ly_i(1-\om) -r$.

For points $Q$ of the upper cylinder $C_u$ -- i.e.\@ the points with $y>1$ -- \eqref{eqn:Bl} states $x_{B^{l}}=x+ l (1+\om) y - l (1+\om) \mod 1$. Therefore, in this case the difference of the \ip s is \[\Delta_{B^l}(Q)=l (y_r + y_i - 1).\]

In summary we have the following results for $\Delta_{B^l}(Q)$:
\begin{equation}
\Delta_{B^l}(Q)=\begin{cases}
ly_i(1-\om) -r \text{ for an } r\in (-1,1) & \text{ if } y\leq 1\\
l (y_r + y_i - 1) & \textnormal{ if } y>1.
\end{cases}\label{eqn:DeltaBl}
\end{equation}

By similar calculations the vertical parabolic element is $A=\matr{1}{0}{w}{1}$ and $A^k$ acts as follows:
\begin{equation}
A^k\circ \PointL{x}{y}=
\begin{cases}
\PointL{x}{y+k x\om \mod \om} & \text{ if } x\leq 1\\
\PointL{x}{y+ k (x\om -\om) \mod 1} & \textnormal{ if } x>1.
\end{cases}\label{eqn:Ak}
\end{equation}

Similarly to the notation $x_{B^l}$
we introduce the notations $y_{A^k}$ for the $y$-coordinate of $A^k\circ Q$ as well as $y_{A^k,r}$ and $y_{A^k,i}$ for the \rp\ and the \ip \ of $y_{A^k}$, respectively. Furthermore we define $\Delta_{A^k}(Q)\coloneqq y_{A^k,i}-y_i$ and obtain for a point $Q=\PointL{x}{y}=\PointL{x_r+x_iw}{y_r+y_iw}\in L_D$ in the left or right cylinder:

\begin{equation}
\Delta_{A^k}(Q)=\begin{cases}
-kx_iw -r \text{ for an } r\in (-1,1) & \text{ if } x\leq 1\\
k(x_r - 1) & \textnormal{ if } x>1.
\end{cases}\label{eqn:DeltaAk}
\end{equation}

\subsection{Points Periodic under $A$ or $B$}\label{subsect:periodic}
To check if a point $Q=\PointL{x,y}=\PointL{x_r+x_iw,y_r+y_iw}\in L_D$ is periodic under $A$ or $B$ one just has to look at its \emph{splitting ratio} in the corresponding vertical respectively horizontal cylinder (cf.\@ \cite{HSGeom}). By this we mean the quotient of the height of the point in the cylinder and the height of the cylinder. It is a well-known fact that $Q$ is periodic under the parabolic element if and only if its splitting ratio is rational. In this section we collect conditions on $Q$ to have a rational splitting ratio and thus to be periodic under $A$ or $B$.

\begin{lem}\label{lem:BPer}
A point $Q=\PointL{x}{y}=\PointL{x_r+x_iw}{y_r+y_iw}\in L_D$ is periodic under the action of $B$ if and only if one of the following two conditions holds:
\begin{enumerate}
\item $y\leq 1$ and $y_i=0$.
\item $y>1$ and $y_r=1-y_i$.
\end{enumerate}
In particular if $Q$ is periodic under $B$ then $0\leq y_i<1$. 
\end{lem}

\begin{proof}
If $y\leq 1$, the point is in the lower cylinder $C_d$ which has height $1$. The height of $Q$ in $C_d$ is $y$, therefore the splitting ratio is $\sr_{C_d}(Q)=\frac{y}{1}=y_r+y_iw$. This is rational if and only if $y_i=0$.

If $y>1$, the point is in the upper cylinder $C_u$ which has height $w-1$. The height of $Q$ in $C_u$ is $y-1$, therefore the splitting ratio is \[\sr_{C_u}(Q)=\frac{y-1}{w-1}=\frac{(y-1)(w+1)}{d-1}=\frac{1}{d-1}(y_r+dy_i-1+(y_r+y_i-1)w).\] This is rational if and only if $y_r=1-y_i$.

For the additional conclusion we observe that solving the inequality $1\leq y_r+y_iw<w$ for $y_r$ and setting $y_r=1-y_i$ yields $1-y_iw\leq 1-y_i<w-y_iw$, which is equivalent to $0\leq y_i < 1.$
\end{proof}

With similar calculations we obtain:
\begin{lem}\label{lem:APer}
A point $Q=\PointL{x}{y}=\PointL{x_r+x_iw}{y_r+y_iw}\in L_D$ is periodic under the action of $A$ if and only if one of the following two conditions holds:
\begin{enumerate}
\item $x\leq 1$ and $x_i=0$.
\item $x>1$ and $x_r=1$.
\end{enumerate}
In particular if $Q$ is periodic under $A$ then $0\leq x_i<1$. 
\end{lem}

\subsection{The Action of $A$ and $B$ in More Detail}\label{sect: technLemm}
To obtain some information about the structure of the Schreier graph from the action of $A^k$ and $B^l$ on a point $Q$, we will analyze how the absolute values of the coordinates change. We are looking for a way to measure the complexity of the connection points. By \autoref{prop: ConnPoints} the connection points are of the form $P=\PointL{x_r+x_iw}{y_r+y_iw}$ with $x_r,x_i,y_r,y_i\in\QQ$. Since the $x$-coordinate and the $y$-coordinate are bounded below by $0$ and above by $w$ and $1+w$, respectively, if the absolute value of an \ip \ grows, in most cases also the corresponding \rp 's absolute value grows. This is why our measure of the complexity is the sum of the \ip 's absolute values, which will be denoted by \[s(Q)\coloneqq |x_i|+|y_i|.\]
Recall that by \autoref{defi:NPN} for any connection point $P$ of the above form $N(P)$ is the least common denominator of the reduced fractions $x_r,x_i,y_r$ and $y_i$. Furthermore, the set of all connection points $P$ with fixed $N(P)=N$ is denoted by $\PN$.

Now we can formulate some quite technical lemmata concerning the growth behavior of connection points $P$ and $s(P)$ under the action of powers of $A$ and $B$. These will help to prove the non-amenability of the Schreier graph of $\SL(L_D)$ modulo $\SL(L_D;P)$ with respect to any finite generating set in \autoref{sect:concl}.

\begin{lem}\label{lem:APerPointsBigger}
If $Q=\PointL{x_r+x_i\om}{y_r+y_i\om}\in \PN$ is periodic under $A$ but not under $B$, there exists an $l_0$, such that for all $l$ with $|l|\geq l_0$ the following inequality holds:
\[s(Q)<s(B^{l}\circ Q).\]
\end{lem}

\begin{lem}\label{lem:BPerPointsBigger}
If $Q=\PointL{x_r+x_i\om}{y_r+y_i\om}\in\PN$ is periodic under $B$ but not under $A$, there exists a $k_0$, such that for all $k\in\ZZ$ with $|k|\geq k_0$ the following inequality holds: \[s(Q)<s(A^{k}\circ Q).\]
\end{lem}

The proofs of these two lemmata are very similar and differ only by some numbers, since the surface $L_D$ is not symmetric with respect to the horizontal and vertical direction. For completeness we will nevertheless prove both lemmata in detail, because they are an important tool in the proof of \autoref{mainresult}.

\begin{proof}[Proof of \autoref{lem:APerPointsBigger}.]
Since $B^{l}$ does not change $y_i$, it increases $s(Q)$ if and only if it increases $|x_i|$. But we know the bounds $0\leq x_i<1$ for a point periodic under $A$ by \autoref{lem:APer}. Hence $s(Q)<s(B^{l}\circ Q)$ is guaranteed if
the absolute value of the change of $x_i$ by $B^{l}$ -- which is $|\Delta_{B^l}(Q)|$ -- is greater than $2$. We will show the existence of $l_0$ for $Q$ in the lower and upper cylinder separately:

If $Q$ is in the lower cylinder $C_d$, \eqref{eqn:DeltaBl} states that $B^{l}$ changes the \ip\ of $x$ by 
	\[\Delta_{B^l}(Q)= ly_i(1-\om)-r \textnormal{ for a } r\in(-1,1).\]
	By the reverse triangle inequality we obtain $|\Delta_{B^l}(Q)|=|ly_i(1-\om)-r|\geq |ly_i(1-\om)|-|r|$ and $|ly_i(1-\om)|\geq 3$ would imply $|\Delta_{B^l}(Q)|>2.$ Since $Q$ is not periodic under $B$, we know $y_i\neq 0$ (\autoref{lem:BPer}) and thus $|y_i|\geq \frac{1}{N}$.
	Thus we can choose \[l_{0,d}=\frac{3N}{\om-1}\] and we obtain
	$|l| \geq \frac{3N}{\om-1}\geq \frac{3}{(\om-1)|y_i|}$ for all $l$ with $|l|\geq l_{0,d}$. This finally implies the estimate $|ly_i(1-\om)|=|l|(\om-1)|y_i|\geq 3$ as required.
							
If $Q$ is in the upper cylinder $C_u$, by \eqref{eqn:DeltaBl} $\Delta_{B^l}(Q)$ equals $l (y_r + y_i - 1)$. Since $Q$ is not periodic under $B$, we know $y_r+y_i-1\neq 0$ (\autoref{lem:BPer}) and thus $|y_r+y_i-1|\geq \frac{1}{N}$. Then \[l_{0,u}=2N+1\] implies $|\Delta_{B^l}(Q)|\geq |l|\frac{1}{N}> 2$ for all $l$ with $|l|\geq l_{0,u}$.

Hence we get the statement of the lemma with \[l_0=\max\{l_{0,d},l_{0,u}\}=\max\left\{\frac{3N}{\om-1},2N+1\right\}.\]
\end{proof}

\begin{proof}[Proof of \autoref{lem:BPerPointsBigger}.]
By analogy with the proof of \autoref{lem:APerPointsBigger} this lemma is certainly true if $|\Delta_{A^k}(Q)|>2$. We will show the existence of $k_0$ for $Q$ in the left and right cylinder separately.

If $Q$ is in the left cylinder $C_l$, \eqref{eqn:DeltaAk} states that $A^{k}$ changes the \ip\ of $y$ by 
	\[\Delta_{A^k}(Q)= -kx_i\om-r \textnormal{ for an } r\in(-1,1).\]
	From the reverse triangle inequality $|\Delta_{A^k}(Q)|=|-kx_i\om-r|\geq |kx_i\om|-|r|$ follows and $|kx_iw|\geq 3$ would imply $|\Delta_{A^k}(Q)|>2.$ Since $Q$ is not periodic under $A$, we know $x_i\neq 0$ (\autoref{lem:APer}) and thus $|x_i|\geq \frac{1}{N}$.
	Thus we can choose \[k_{0,l}=\frac{3N}{w}\] and obtain
	$|k| \geq \frac{3N}{w}\geq \frac{3}{|x_i|w}$ for all $k$ with $|k|\geq k_{0,l}$. This implies
	$|kx_iw|=|k|w|x_i|\geq 3$ as desired.
	
	If $Q$ is in the right cylinder, $\Delta_{A^k}(Q)=k (x_r - 1)$ by \eqref{eqn:DeltaAk}. Since $Q$ is not periodic under $A$, we know $x_r-1\neq 0$ (\autoref{lem:APer}) and thus $|x_r-1|\geq \frac{1}{N}$. Then \[k_{0,r}=2N+1\] implies for all $k$ with $|k|\geq k_{0,r}$, that $|\Delta_{A^k}(Q)|\geq |k|\frac{1}{N}> 2$ holds.

Hence we get the statement of the lemma with \[k_0=\max\{k_{0,l},k_{0,r}\}=\max\left\{\frac{3N}{w},2N+1\right\}.\]
\end{proof}

The next two lemmata will help us to prove \autoref{lem:DreiVonVierLemma}. An important fact in the proof is that $|a\pm b|>|b|$ implies $\sgn(a\pm b)=\sgn(a)$.

\begin{lem}\label{lem:pmdifA}
If $Q=\PointL{x_r+x_i\om}{y_r+y_i\om}\in\PN$ is a point not periodic under $A$, then the signs of $\Delta_{A^k}(Q)$ and $\Delta_{A^{-k}}(Q)$ are different for all $k>k_0$.
\end{lem}

\begin{proof}
If $Q$ is in the right cylinder, $\Delta_{A^{\pm k}}(Q)=\pm k(x_r-1)$. Since $Q$ is not periodic under $A$, we have $x_r\neq 1$ and  $k(x_r-1)$ and $-k(x_r-1)$ have different signs.

If the point $Q$ is in the left cylinder, $\Delta_{A^{+k}}(Q)=-kx_i\om-r$ for some $r\in(-1,1)$ and $\Delta_{A^{-k}}(Q)=kx_i\om+r'$ for an $r'\in(-1,1)$. As seen in the proof of \autoref{lem:APerPointsBigger}, the absolute value of these changes is greater than $2$ for $k>k_0$, hence $|-kx_i\om-r|$ and $|kx_i\om+r'|$ are bigger than $|r|$ and $|r'|$. Thus the first term -- $\Delta_{A^{+k}}(Q)$ -- has sign $\sgn(-kx_i)$, which is not $0$ because $Q$ is not periodic under $A$, whereas the second one -- $\Delta_{A^{-k}}(Q)$ -- has sign $\sgn(kx_i)=-\sgn(-kx_i)$.
\end{proof}

\begin{lem}\label{lem:pmdifB}
If $Q=\PointL{x_r+x_i\om}{y_r+y_i\om}\in\PN$ is a point not periodic under $B$, then the signs of $\Delta_{B^l}(Q)$ and $\Delta_{B^{-l}}(Q)$ are different for all $l>l_0$.
\end{lem}

We will skip the proof of this lemma, since it works completely analogically to the proof of \autoref{lem:pmdifA}.

The last and most important lemma concerning the behavior of $s(Q)$ under the action of powers of $A$ and $B$ is the following:
\begin{lem}\label{lem:DreiVonVierLemma}
Let $Q=\PointL{x_r+x_i\om}{y_r+y_i\om}\in\PN$ be a point periodic neither under $A$ nor under $B$. Then there are numbers $k_1$ and $l_1$ such that for all pairs $(k,l)$ with $k>k_1$ and $l>l_1$ at least three of the following four inequalities hold:
\begin{align*}
s(Q) &< s(A^{k}\circ Q),\\
s(Q) &< s(A^{-k}\circ Q),\\
s(Q) &< s(B^{l}\circ Q),\\
s(Q) &< s(B^{-l}\circ Q).
\end{align*}
\end{lem}

\begin{proof}
As a first observation, we see that since $A$ does not change $x_i$ and $B$ does not change $y_i$, we only have to look at the effect of $A^{\pm k}$ on $|y_i|$ and of $B^{\pm l}$ on $|x_i|$ to see the effect on $s(Q)=|x_i|+|y_i|$.

The second step is to apply \autoref{lem:pmdifA} and \autoref{lem:pmdifB}: Since the changes of $y_i$ by $A^k$ and by $A^{-k}$ have different signs, one of them has the same sign as $y_i$ itself and thus the corresponding $A^{\pm k}$ increases $|y_i|$. By analogy one of $B^{\pm l}$ has $\sgn{\Delta_{B^{\pm l}}(Q)}=\sgn(x_i)$ and thus increases $|x_i|$.
Hence at least two of the four inequalities hold. For the third one we have to show that one of the two remaining $\Delta$'s is big enough to increase $|y_i|$ respectively $|x_i|$ even though their sign is different from the sign of $y_i$ respectively $x_i$. This means it has to be greater than $2|y_i|$ respectively $2|x_i|$.

If $|x_i|\geq |y_i|$ this will be $\Delta_{A^k}(Q)$ for $k$ with $|k|$ big enough, otherwise it will be $\Delta_{B^l}(Q)$ for $l$ with $|l|$ big enough. To finish the proof, we distinguish these two cases.

{\bf First Case} $|x_i|\geq |y_i|$: Let us first find a $k_1$ with the property that 
for all $k\in \ZZ$ with $|k|>k_1$ the inequality $|\Delta_{A^k}(Q)|>2|x_i|$ holds. In particular for points with $|x_i|\geq |y_i|$ this inequality implies $|y_{A^k,i}|=|y_i+\Delta_{A^k}(Q)|>|x_i|\geq |y_i|$.

We show this for $Q$ in the left or in the right cylinder seperately:

If $Q$ is in the left cylinder, $\Delta_{A^k}(Q)=-kx_i\om-r$ for an $r\in(-1,1)$. With the reversed triangle inequality we get $|-kx_i\om-r|>2|x_i|\textnormal{, if } |k| |x_i| \om > 2|x_i|+|r|$. This again would be implied by $|k| |x_i| \om \geq 2 |x_i| + 1$, which is equivalent to $|k|\geq \frac{2}{\om}+\frac{1}{|x_i|\om}$. Since $Q$ is not periodic under $A$, we know (from \autoref{lem:APer}) that $x_i\neq 0$ and thus $|x_i|\geq \frac{1}{N}$. Hence the above inequality is fulfilled if $|k|\geq \frac{2+N}{\om}$.

If $Q$ is in the right cylinder, the $x$-coordinate satisfies $1<x_r+x_iw<1+w$ and therefore there is a $q\in(0,\om)$, such that $x_r=-x_i\om+1+q$. The change of $y_i$ is $\Delta_{A^k}(Q)=k(x_r-1)$. Using the reversed triangle inequality we obtain $|\Delta_{A^k}|=|k| \cdot |-x_i\om+q|\geq |k| (|x_i\om|-|q|)>|k|(|x_i|\om-\om)$. Hence $|\Delta_{A^k}(Q)|>2|x_i|$, if $|k|\geq \frac{2}{\om}\cdot\frac{|x_i|}{|x_i|-1}$. This term decreases for $|x_i|>1$.

The smallest possible value for $|x_i|>1$ is $|x_i|=1+\frac{1}{N}$. Thus for points $Q$ with $|x_i|>1$ the inequality $|\Delta_{A^k}(Q)|>2|x_i|$ is guaranteed, if $|k|\geq\frac{2(N+1)}{\om}$.

For points $Q$ with $|x_i|\leq 1$ already $|k|>k_0$ suffices, because the same computations as in the proof of \autoref{lem:BPerPointsBigger} show $|\Delta_{A^k}(Q)|>2\geq 2 |x_i|$ for all points $Q$ not periodic under $A$ and all $k$ with $|k|>k_0$.

Summarizing, $|\Delta_{A^k}(Q)|>2|x_i|$ holds for all $k$ with $|k|>k_1$, where \[k_1= \max\left\{\frac{2+N}{\om},k_0,\frac{2(N+1)}{\om}\right\},\] where $\frac{2+N}{\om}$ can be omitted because it is always smaller than $\frac{2(N+1)}{\om}$.

{\bf Second Case} $|x_i|< |y_i|$: we have to find an $l_1$ satisfying $|\Delta_{B^l}(Q)|>2|y_i|$ for all $l\in \ZZ$ with $|l|>l_1$. In particular, for points with $|x_i|< |y_i|$ this inequality implies $|x_{B^l,i}|=|x_i+\Delta_{B^l}(Q)|>|y_i|> |x_i|$.

We show this for $Q$ in the lower or in the upper cylinder seperately:

If $Q$ is in the lower cylinder, $|\Delta_{B^l}|$ is greater than $|ly_i(1-\om)|-1$ by the same arguments as in the proof of \autoref{lem:APerPointsBigger}. Therefore, it is greater than $2|y_i|$, if $|l|(\om-1)|y_i|\geq 2|y_i|+1$, and this is equivalent to $|l|\geq \frac{2}{\om-1}+\frac{1}{(\om-1)|y_i|}$. Since $Q$ is in the lower cylinder but not periodic under $B$, we know that $y_i\neq 0$ and thus $|y_i|\geq \frac{1}{N}$. Hence the inequality holds if $|l|\geq\frac{2+N}{\om-1}$.

If $Q$ is in the upper cylinder, the $y$-coordinate satisfies $1<y_r+y_iw<w$ and therefore there is a $q\in(0,\om-1)$, such that $y_r=-y_i\om+1+q$. The change of $x_i$ is $\Delta_{B^l}(Q)=l(y_r+y_i-1)$. Again, the reversed triangle inequality implies $|\Delta_{B^l}|=|l| \cdot |(y_i(1-\om)+q)|\geq |l| (|y_i(1-\om)|-|q|)>|l|((\om-1)|y_i|-(\om-1))$. Hence $|\Delta_{B^l}(Q)|>2|y_i|$, if $|l|\geq \frac{2}{\om-1}\cdot \frac{|y_i|}{|y_i|-1}$. This term is decreasing for $|y_i|>1$.

The smallest possible value for $|y_i|>1$ is $|y_i|=1+\frac{1}{N}$. Thus for points $Q$ with $|y_i|>1$ the inequality $|\Delta_{B^l}(Q)|>2|y_i|$ is guaranteed, if $|l|\geq\frac{2(N+1)}{\om-1}$.
For points $Q$ with $|y_i|\leq 1$, already $|l|>l_0$ suffices, because the same computations as in the proof of \autoref{lem:APerPointsBigger} show $|\Delta_{B^l}(Q)|>2\geq 2 |y_i|$ for all points $Q$ not periodic under $B$ and all $l$ with $|l|>l_0$.

Summarizing $|\Delta_{B^l}(Q)|>2|y_i|$ holds for all $l$ with $|l|>l_1$, where \[l_1= \max\{\frac{2+N}{\om-1},l_0,\frac{2(N+1)}{\om-1}\},\] where $\frac{2+N}{\om-1}$ can be omitted because it is always smaller than $\frac{2(N+1)}{\om-1}$.

All in all we proved that for all points $Q$ with $|x_i|\geq |y_i|$ the following inequalities hold for all pairs $(k,l)$ with $k>k_1$ and $l>l_1$:
\begin{align*}
 & s(Q) < s(A^{k}\circ Q),\\
 & s(Q) < s(A^{-k}\circ Q),\\
\textnormal{at least one of } & s(Q) < s(B^{l}\circ Q) \textnormal{ and } s(Q) < s(B^{-l}\circ Q).
\end{align*}

For points $Q$ with $|x_i|< |y_i|$ the following inequalities hold for all pairs $(k,l)$ with $k>k_1$ and $l>l_1$:
\begin{align*}
 & s(Q) < s(B^{l}\circ Q),\\
 & s(Q) < s(B^{-l}\circ Q),\\
\textnormal{at least one of } & s(Q) < s(A^{k}\circ Q) \textnormal{ and } s(Q) < s(A^{-k}\circ Q).
\end{align*}
\end{proof}

\subsection{Analogous Results for $L_{D,\pm 1}$}\label{subs: LDpm}
As mentioned above all results of this section also hold for the surfaces $L_{D,+1}$ and $L_{D,-1}$ -- one just has to replace some numbers and values. We now list these numbers.
\subsubsection{The Surface $L_{D,+1}$}
The vertical respectively horizontal elements of $\SL(L_{D,+1})$ are $A_+=\smatr{1}{0}{w-1}{1}$ respectively $B_+=\smatr{1}{1+w}{0}{1}$. A point $Q=\PointL{x}{y}=\PointL{x_r+x_iw}{y_r+y_iw}\in L_{D,+1}$ is periodic under the action of $A_+$ if and only if one of the following two conditions holds:
\begin{enumerate}
\item $x\leq 1$ and $x_i=0$.
\item $x>1$ and $x_r=1$.
\end{enumerate}
In particular if $Q$ is periodic under $A_+$ then $0\leq x_i<1$. 

The point $Q\in L_{D,+1}$ is periodic under the action of $B_+$ if and only if one of the following two conditions holds:
\begin{enumerate}
\item $y\leq 1$ and $y_i=0$.
\item $y>1$ and $y_r+2y_i=1$.
\end{enumerate}
In particular if $Q$ is periodic under $B_+$ then $0\leq y_i<1$.

The analogous equation to \eqref{eqn:Ak} is
\begin{equation*}
A_+^k\circ \PointL{x}{y}=
\begin{cases}
\PointL{x}{y+k x(\om-1) \mod (\om-1)} & \text{ if } x\leq 1\\
\PointL{x}{y+ k (x-1)(w-1) \mod 1} & \textnormal{ if } x>1.
\end{cases}
\end{equation*}

This yields 
\begin{equation*}
\Delta_{A_+^k}(Q)=\begin{cases}
-kx_iw -r \text{ for an } r\in (-1,1) & \text{ if } x\leq 1\\
k(x_r  - 1) & \textnormal{ if } x>1.
\end{cases}
\end{equation*}

Accordingly, \eqref{eqn:Bl} becomes
\begin{equation*}
B_+^l\circ \PointL{x}{y}=
\begin{cases}
\PointL{x+l y(1+w) \mod (1+w)}{y} & \text{ if } y\leq 1\\
\PointL{x+ l (y-1)(1+w) \mod 1}{y} & \textnormal{ if } y>1.
\end{cases}
\end{equation*}

This yields 
\begin{equation*}
\Delta_{B_+^l}(Q)=\begin{cases}
-ly_i(2-w) -r \text{ for an } r\in (-1,1) & \text{ if } y\leq 1\\
l(2y_i+y_r - 1) & \textnormal{ if } y>1.
\end{cases}
\end{equation*}

For $L_{D,+1}$ the numbers $k_0,l_0,k_1$ and $l_1$ from \autoref{lem:APerPointsBigger} to \autoref{lem:DreiVonVierLemma} are 
\begin{align*}
k_0&=\max\left\{\frac{3N}{w},2N+1\right\}, \\
l_0&=\max\left\{\frac{3N}{2-w},2N+1\right\},\\
k_1&=\max\left\{\frac{2+N}{w},k_0,\frac{2(N+1)}{w}\right\}, \\
l_1&= \max\{\frac{2+N}{w-2},l_0,\frac{2(N+1)}{w-2}\}.
\end{align*}

\subsubsection{The Surface $L_{D,-1}$}
The vertical respectively horizontal elements of $\SL(L_{D,-1})$ are $A_-=\smatr{1}{0}{w}{1}$ respectively $B_-=\smatr{1}{w}{0}{1}$. A point $Q=\PointL{x}{y}=\PointL{x_r+x_iw}{y_r+y_iw}\in L_{D,-1}$ is periodic under the action of $A_-$ if and only if one of the following two conditions holds:
\begin{enumerate}
\item $x\leq 1$ and $x_i=0$.
\item $x>1$ and $x_r+x_i=1$.
\end{enumerate}
In particular if $Q$ is periodic under $A_-$ then $0\leq x_i<1$. 

The point $Q\in L_{D,-1}$ is periodic under the action of $B_-$ if and only if one of the following two conditions holds:
\begin{enumerate}
\item $y\leq 1$ and $y_i=0$.
\item $y>1$ and $y_r+y_i=1$.
\end{enumerate}
In particular if $Q$ is periodic under $B_-$ then $0\leq y_i<1$.

The analogous equation to \eqref{eqn:Ak} is
\begin{equation*}
A_-^k\circ \PointL{x}{y}=
\begin{cases}
\PointL{x}{y+k x w \mod w} & \text{ if } x\leq 1\\
\PointL{x}{y+ k (x-1)w \mod 1} & \textnormal{ if } x>1.
\end{cases}
\end{equation*}

This yields 
\begin{equation*}
\Delta_{A_-^k}(Q)=\begin{cases}
-kx_i(w-1) -r \text{ for an } r\in (-1,1) & \text{ if } x\leq 1\\
k(x_r + x_i  - 1) & \textnormal{ if } x>1.
\end{cases}
\end{equation*}

Accordingly, \eqref{eqn:Bl} becomes
\begin{equation*}
B_-^l\circ \PointL{x}{y}=
\begin{cases}
\PointL{x+l y w \mod w}{y} & \text{ if } y\leq 1\\
\PointL{x+ l (y-1)w \mod 1}{y} & \textnormal{ if } y>1.
\end{cases}
\end{equation*}

This yields 
\begin{equation*}
\Delta_{B_-^l}(Q)=\begin{cases}
-ly_i(w-1) -r \text{ for an } r\in (-1,1) & \text{ if } y\leq 1\\
l(y_r+y_i - 1) & \textnormal{ if } y>1.
\end{cases}
\end{equation*}

For $L_{D,-1}$ the numbers $k_0,l_0,k_1$ and $l_1$ from \autoref{lem:APerPointsBigger} to \autoref{lem:DreiVonVierLemma} are 
\begin{align*}
k_0&=l_0=\max\left\{\frac{3N}{w-1},2N+1\right\}, \\
k_1&=l_1=\max\{\frac{2+N}{\om-1},l_0,\frac{2(N+1)}{\om-1}\}. \\
\end{align*}


\section{Veech Groups with $\frac{1}{2}<\delta< 1$}\label{sect:concl}
After we provided the auxiliary material in the previous sections in this section we can reap the fruit of our labor and prove our main theorems:

\begin{prop}\label{mainprop}
For every non-periodic connection point $P$ on the Veech surface $L_{D,\epsilon}$ (with $D$ not a square) the critical exponent of the (infinitely generated) stabilizer subgroup $\SL(L_{D,\epsilon};P):=\Stab_{\SL(L_{D,\epsilon})}(P)$ is strictly between $\frac{1}{2}$ and $1$.
\end{prop}

As already mentioned in the introduction by \autoref{prop:commens} the Veech groups of affine coverings of $L_{D,\epsilon}$ branched at the singularity and $P$ are commensurable to $\SL(L_{D,\epsilon};P)$. By \autoref{prop: commenssameCE} commensurable Fuchsian groups have the same critical exponent and \autoref{mainresult} follows:
{
\renewcommand{\themthm}{\ref{mainresult}}
\begin{mthm}
There exist translation surfaces whose Veech group is infinitely generated with critical exponent strictly between $\frac{1}{2}$ and $1$. More precisely this is the case for every affine covering of $L_{D,\epsilon}$ (with $D$ not a square) branched at the singularity and one non-periodic connection point $P$.
\end{mthm}
\addtocounter{mthm}{-1}
}

In \autoref{critexp} we have seen, that $\delta(\SL(L_{D,\epsilon};P))>\frac{1}{2}$, since $\SL(L_{D,\epsilon};P)$ is non-elementary and contains a parabolic element (\autoref{thm:BeaPat}). Furthermore, by \autoref{prop: Tapie} the critical exponent $\delta(\SL(L_{D,\epsilon};P))$ is strictly smaller than $1$ if the Schreier graph of $\SL(L_{D,\epsilon})$ with respect to $\SL(L_{D,\epsilon};P)$ and any finite generating set $S$ of $\SL(L_{D,\epsilon})$ is non-amenable. 

So let us now show that for every non-periodic connection point $P$ and any finite set $S$ generating $\Gamma=\SL(L_{D,\epsilon})$ the graph $G_{\Gamma,\Pi,S}$ is non-amenable, where in order to improve readability $\Pi$ denotes the subgroup $\SL(L_{D,\epsilon};P)$. \autoref{IndependFromGS} tells us that we can change the finite generating set. Thus $G_{\Gamma,\Pi,S}$ is non-amenable if and only if $G=G_{\Gamma,\Pi,S\cup\{A^k,B^l\}}$ is non-amenable for some numbers $k$ and $l$.

Let $G'$ be the subgraph of $G$ induced by all vertices of $G$ that correspond to points that are not periodic under both $A$ and $B$. These points are described in \autoref{subsect:periodic}. In particular points periodic under $A$ and $B$ have bounded \ip s $x_i$ and $y_i$. By \autoref{thm: orbitssameN} the $\Gamma$-orbit of $P$ is contained in $\PN[(P)]$. Hence the denominators of the components of points in $\Gamma\circ P$ are bounded by $N(P)$. Therefore, there are only finitely many points periodic under $A$ and $B$. Thus we can apply  \autoref{PerPoiWegl} and non-amenability of $G'$ implies non-amenability of $G$.

The last step to simplify the graph is to omit all edges that are not labeled with $A^k$ or $B^l$. We call this graph $G''$. Omitting edges does not increase the Cheeger constant (\autoref{omitedges}). Hence the graphs $G'$, $G$ and $G_{\Gamma,\Pi,S}$ are non-amenable if $G''$ is non-amenable.

That $G''$ is indeed non-amenable follows from:

\begin{prop}\label{ZweiMoeglZshgkomp}
There are numbers $k$ and $l$ depending only on $N\coloneqq N(P)$ and $\om$, such that every connected component of $G''$ is either isomorphic to the infinite $4$-valent tree or to the \gz \ (cf. \autoref{pic:F2BaumRL4T}).
\end{prop}

\begin{figure}[h!]
\begin{center}
\includegraphics[width=0.9\linewidth]{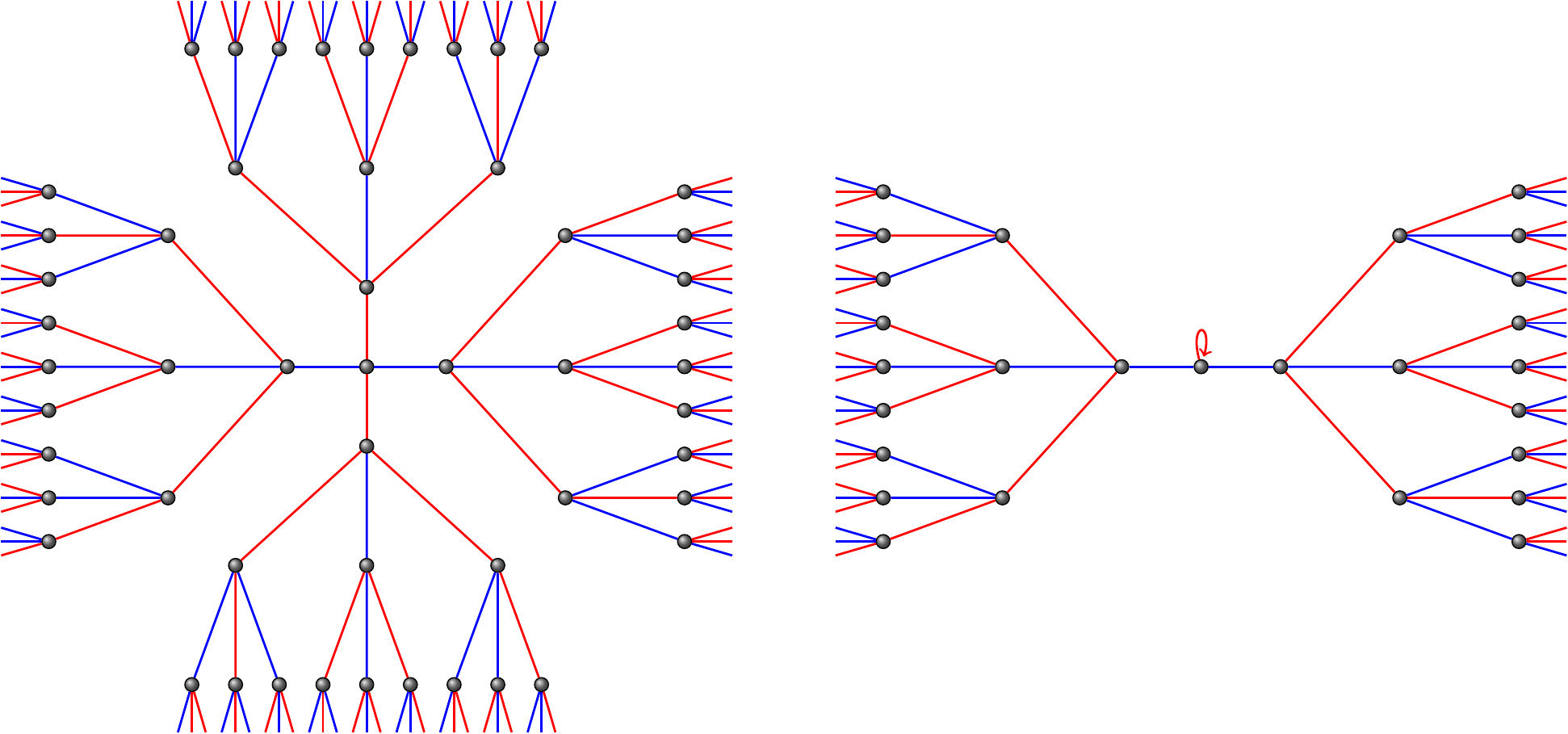}
\caption[Infinite $4$-valent tree and root-looped $4$-valent tree]{The infinite $4$-valent and the root-looped $4$-valent tree.}
\label{pic:F2BaumRL4T}
\end{center}
\end{figure}

As seen in \autoref{CayleyFk} and \autoref{cheegerconstRL4T} both graphs have Cheeger constant $\frac{2}{3}$, hence this proposition together with Corollary \ref{connected} implies the non-amenability of $G''$ and thus gives also:
\begin{cor}
The graph $G_{\{\Gamma,\Pi,S\}}$ is non-amenable for any nonperiodic connection point $P$ and any finite generating set $S$ of $\Gamma$.
\end{cor}

\begin{proof}[Proof of \autoref{ZweiMoeglZshgkomp}.]
In this proof we will again concentrate on the surfaces $L_D$. For the surfaces $L_{D,\pm 1}$ nearly the same proof works. One just has to replace the values $k_0,l_0,k_1$ and $l_1$ with the corresponding values described in \autoref{subs: LDpm}.

We will analyze the points corresponding to the vertices of $G''$ and the sum of the absolute values of their \ip s. Recall that for a point $Q=\PointL{x_r+x_i\om}{y_r+y_i\om}$ this sum $|x_i|+|y_i|$ is denoted by $s(Q)$. 
\begin{enumerate}
\item\label{ZweiVonZwei} For every point $Q$ in the orbit of $P$ periodic under $A$ (resp. $B$) we know $N(Q)=N(P)\eqqcolon N$ and this is a multiple of the period (see \autoref{subsect:periodic}). Hence in $G''$ the circle at the periodic point $Q$ is in fact a loop -- i.e. of length $1$ -- if we choose $k$ (resp. $l$) to be a multiple of $N$.

Thus $Q$ has exactly two neighbors $\neq Q$, namely $B^l\circ Q$ and $B^{-l}\circ Q$ (resp. ($A^k\circ Q$ and $A^{-k}\circ Q$)). \autoref{lem:APerPointsBigger} and \autoref{lem:BPerPointsBigger} state that for $l>l_0=\max\left\{\frac{3N}{\om-1},2N+1\right\}$ \[s(Q)<s(B^{l}\circ Q) \textnormal{ and } s(Q)<s(B^{-l}\circ Q)\] (resp. for $k>k_0=\max\left\{\frac{3N}{w},2N+1\right\}$ the analogous statement for $s(A^{\pm k}\circ Q)$) holds.

\item \label{DreiVonVier} In \autoref{lem:DreiVonVierLemma} we have seen that if we choose $k>k_1=\max\left\{k_0,\frac{2(N+1)}{\om}\right\}$ and $l>l_1=\max\left\{l_0,\frac{2(N+1)}{\om-1}\right\}$, for all points $Q$ neither periodic under $A$ nor under $B$ at least three of the following inequalities hold:
\begin{align*}
s(Q) &< s(A^{k}\circ Q),\\
s(Q) &< s(A^{-k}\circ Q),\\
s(Q) &< s(B^{l}\circ Q),\\
s(Q) &< s(B^{-l}\circ Q).
\end{align*}
\end{enumerate}
All these conditions on the choice of $k$ and $l$ only depend on $N$ and $w$. Hence we can and do choose $k$ as the smallest multiple of $N$ bigger than $k_1$ and $l$ as the smallest multiple of $N$ bigger than $l_1$ and analyze which graphs can occur as connected components of $G''$:

\textbf{First Case: in the connected component there is a point $Q$ periodic under $A$ or $B$:}

Let us remove the loop at $Q$ and look at an path of length $n$ starting in $Q=Q_0$ visiting the points \mbox{$Q_0,Q_1,\ldots Q_n$} with $Q_1\neq Q_0$ and $Q_{j+2}\neq Q_j$. From \autoref{ZweiVonZwei} we know that $s(Q_1)>s(Q)$. Thus $Q_1$ has a neighbor with smaller $s$-value and can not be periodic under $A$ or $B$. Furthermore we also know the one and only neighbor of $Q_1$ that has a lower $s$-value than $Q_1$ is $Q_0$ (because of \autoref{DreiVonVier}) and we get $s(Q_2)>s(Q_1)$. The same argument for the following points show that none of the points $Q_j$, $j\geq 1$ is periodic under $A$ or $B$ and $\left(s(Q_j)\right)_{j=0,\ldots ,n}$ is strictly increasing. In particular there is no circle in this connected component, all vertices beside $Q$ in the same component have valency $4$. This component is isomorphic to the \gz .

\textbf{Second Case: in the connected component there is no point periodic under $A$ or $B$, so all vertices have valency $4$:}

Suppose in the connected component there is a circle $[v_0,v_1,\ldots, v_{n-1},v_0]$ with corresponding points $[Q_0,Q_1,\ldots, Q_{n-1},Q_0]$. Then the set $\{s(Q_0),\ldots, s(Q_{n-1})\}$ is finite and thus contains its maximum. Let $Q_j$ be a point with $s(Q_j)$ this maximum. This implies $s(Q_{j-1})\leq s(Q_j)$ and $s(Q_{j+1})\leq s(Q_j)$, which is a contradiction to \autoref{DreiVonVier}. Thus in this case the connected component is the infinite $4$-valent tree.
\end{proof}


\section{Orbits of Connection Points of $L_8$}\label{sect:orbits}
In \autoref{thm: orbitssameN} we have seen that for a connection point $P$ the least common denominator $N(P)$ is an invariant for the $\SL(L_{D,\epsilon}$-orbit of $P$. Since all points with both coordinates in $\QQ(w)$ are connection points this implies that there are infinitely many distinct orbits of connection points. We conjecture that for all the surfaces $L_{D,\epsilon}$ every fixed number $N$ the set $\PN$ decomposes into a \textbf{finite} number of $\SL(L_{D,\epsilon})$-orbits. In this section we will concentrate on the case $D=8$. We establish a method to find lower bounds and give upper bounds on the number of orbits for fixed least common denominator $N$. In particular, we prove this conjecture for $D=8$.

First we define a finite graph $G_N$ for each $N$ such that the number of connected components $C(N)$ of this graph is a lower bound on the number of orbits. The number $C(N)$ for small values $N$ is listed in \autoref{tab:ConnComp}.

\begin{table}[h!]
\begin{center}
\begin{tabular}{c|c|c|c|c|c|c|c|c|c|c|c|c|c|c}
\rule[-1ex]{0pt}{2.5ex}$N$ & 1 & 2 & 3 & 4 & 5 & 6 & 7 & 8 & 9 & 10 & 11 & 12 & 13 & 14\\ 
\hline 
\rule[-1ex]{0pt}{2.5ex} $C(N)$& 1 & 5 & 1 & 8 & 1 & 5 & 3 & 8 & 1 & 5 & 1 & 8 & 1 & 15\\ 
\hline  \hline
\rule[-1ex]{0pt}{2.5ex} $N$ & 15 & 16 & 17 & 18 & 19 & 20 & 21 & 22 & 23 & 24 & 25 & 26 & 27 & 28\\ 
\hline 
\rule[-1ex]{0pt}{2.5ex} $C(N)$& 1 & 8 & 3 & 5 & 1 & 8 & 3 & 5 & 3 & 8 & 1 & 5 & 1 & 24\\ 
\end{tabular}
\caption{The number of connected components $C(N)$ of the graph $G_N$.}
\label{tab:ConnComp}
\end{center}
\end{table}

\begin{conj}
The values of $C(N)$ for small $N$ shown in \autoref{tab:ConnComp} indicate that the function $C:N\mapsto C(N)$ is weakly multiplicative, i.e. $C(NM)=C(N)C(M)$ for all co-prime natural numbers $N$ and $M$. Furthermore, we conjecture that, if $N=p^k$ is a power of an odd prime, $C(N)$ equals $C(p)$.
\end{conj}

Crucial for the method for the lower bound is the fact that $\SL(L_8)$ is generated by the two parabolic elements $A=\smatr{1}{0}{\sqrt{2}}{1}$ and $B=\smatr{1}{\sqrt{2}}{0}{1}$ analyzed in \autoref{subs:AandB}.

The main result of this section is:

{
\renewcommand{\thetheorem}{\ref{thm:Orbits}}
\begin{thm}
Set $\om\coloneqq \sqrt{2}$ and fix $N\in\NN$. The set $\PN$ decomposes into a \textbf{finite} number of $\SL(L_8)$-orbits.
\end{thm}
\addtocounter{theorem}{-1}
}

For the proof we give a ``continued fraction''-like algorithm that uses the elements $A=\smatr{1}{0}{\sqrt{2}}{1}$ and $B=\smatr{1}{\sqrt{2}}{0}{1}$ to connect every connection point with this fixed $N$ with a point of a finite set. A similar algorithm should also work for $D=5$, because also the Veech group $\SL(L_5)$ is generated by the vertical and horizontal parabolic elements. For bigger $D$ it might also be true that there are only finitely many $\SL(L_D)$-orbits of connection points with fixed $N$, but the subgroup $\langle A,B\rangle < \SL(L_D)$ is too small for the algorithm to still work.

One can easily check that the action of $A$ and $B$ (see \autoref{subs:AandB}) can be described as follows:

$A^{\pm 1} \circ \PointL{x_r+x_i\sqrt{2}}{y_r+y_i\sqrt{2}}=\PointL{x_r+x_i\sqrt{2}}{y_r'+y_i'\sqrt{2}}$ with
\begin{align*}
y_r' &= y_r\pm 2x_i\mp t_{A,r} \textnormal{ with } t_{A,r}\in\{0,1,2\}\\
y_i' &= y_i\pm x_r\mp t_{A,i} \textnormal{ with } t_{A,i}\in\{0,1\}
\end{align*}

$B	^{\pm 1} \circ \PointL{x_r+x_i\sqrt{2}}{y_r+y_i\sqrt{2}}=\PointL{x_r'+x_i'\sqrt{2}}{y_r+y_i\sqrt{2}}$ with
\begin{align*}
x_r' &= x_r\pm y_r \pm 2y_i\mp t_{B,r} \textnormal{ with } t_{B,r}\in\{0,1,2\}\\
x_i' &= x_i\pm y_r \pm y_i \mp t_{B,i} \textnormal{ with } t_{B,i}\in\{0,1\}
\end{align*}

We start with finding a lower bound on the number of orbits in $\PN$ with fixed $N$. We define the following equivalence relation on the points $P=\PointL{x_r+x_i\sqrt{2}}{y_r+y_i\sqrt{2}}$ with the four components $x_r,x_i,y_r$ and $y_i$ reduced fractions with least common denominator $N$. For all such points expand the four components to fractions with $N$ as denominator. Then two points are equivalent, if and only if the two vectors consisting of the numerators modulo $N$ are the same. This yields the set \[V=\{[a,b,c,d]\in \left(\quotient{\ZZ}{N\ZZ}\right)^4 \mid \gcd(a,b,c,d,N)=1\}\] as quotient. Since all the $t_{A/B,r/i}$ in the description of the action of $A$ and $B$ are integers, adding them corresponds to adding multiples of $N$ to the numerators of the expanded fractions and the action descends to an action on $V$. Let $G_N$ be the graph describing this action, i.e. the following graph:
\begin{itemize}
\item The vertex set is $V=\{[a,b,c,d]\in \left(\quotient{\ZZ}{N\ZZ}\right)^4 \mid \gcd(a,b,c,d,N)=1\}$.
\item The edges $E$ are labeled with $A$ or $B$ and go from $v\in V$ to $v'\in V$, if and only if for any point $P$ represented by $v$ the point $A\circ P$ respectively $B\circ P$ is represented by $v'$.
\end{itemize}

Obviously, the number of connected components $C(N)$ of this graph $G_N$ is a lower bound on the number of orbits of the action of $\SL(L_8)$ for fixed $N$, since $\SL(L_8)=\langle A,B\rangle$. The graph $G_2$ is illustrated in \autoref{pic:ZweiZweiGraph}.

\begin{figure}[h!]
\begin{center}
\includegraphics{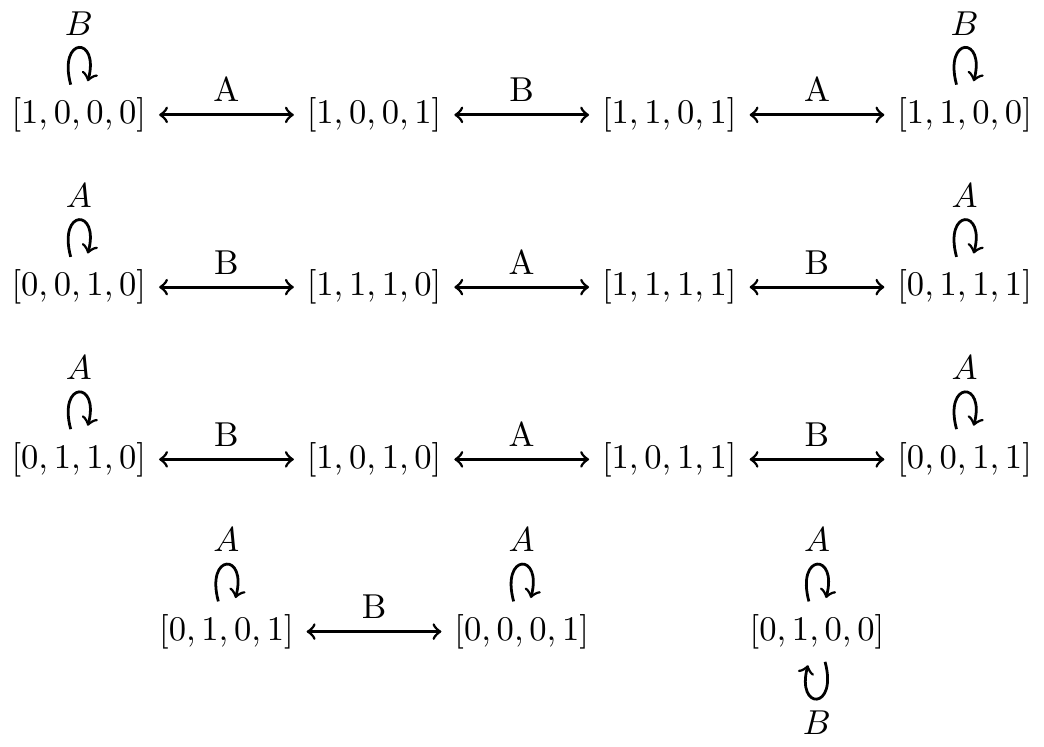}
\caption{The graph $G_2$ with $5$ connected components.}
\label{pic:ZweiZweiGraph}
\end{center}
\end{figure}

\begin{rem}
The action on $V$ can be described as follows:
\begin{align*}
A\circ [a,b,c,d]&=[a,b,c+2b,d+a] \mod N\\
B\circ [a,b,c,d]&=[a+c+2d,b+c+d,c,d] \mod N
\end{align*}
So the edges can be computed very easily, but notice that the vertex set $V$ grows quite fast for growing $N$: it has approximately $N^4$ elements.
\end{rem}

Let us now prove the main result of this section.

\begin{proof}[Proof of \autoref{thm:Orbits}.]
Consider the following set of points of $L_8$: \[S=\left\{\PointL{x_r+x_i\om}{y_r+y_i\om}\in L_8 \mid |x_i|,|y_i|\leq \Grenze\right\}.\] 

Since $N$ is fixed, there are only finitely many possible values of $x_i$ and $y_i$. Moreover, the fact that $0\leq x=x_r+x_iw < 1+w$ and $0\leq y=y_r+y_iw < w$ implies that also the number of possible values of $x_r$ and $y_r$ is bounded. Thus the set $S$ is finite.

In \autoref{alg:Alg} we describe how to find a word $W$ in $A^\pm$ and $B^\pm$ such that $W\circ P\in S$, when a non-periodic point $P=\PointL{x_r+x_iw}{y_r+y_iw}\in \PN$ is given as input. Note that there are only $6$ periodic points (\autoref{thm:periodic}) and all of them already belong to $S$. Then $|S|$ is an upper bound on the number of orbits for fixed $N$.

\begin{algorithm}[h!]
\SetAlgoLined
\KwIn{A non-periodic point $P=\PointL{x_r+x_iw}{y_r+y_iw}$ with $x_r,x_i,y_r,y_i$ reduced fractions with least common denominator $N$}
\KwOut{A word $W$ in $A^\pm$ and $B^\pm$ with $W\circ P\in S$}
$W\leftarrow$ empty word\;
\While{$P\notin S$}{
	\lIf{$P$ is periodic under $B$\tcp*{Case $1$}}{
	\Indp$P\leftarrow A\inv B\inv A \circ P$, \quad
	$W \leftarrow \concat(A\inv B\inv A, W)$\;
	}
	\Indm\lElseIf{$P$ is periodic under $A$\tcp*{Case $2$}}{
	\Indp$P\leftarrow B\inv A\inv B \circ P$, \quad
	$W \leftarrow \concat(B\inv A\inv B, W)$\;
	}\Indm
	\lElseIf{$|x_i|<|y_i|$ \tcp*{Case $3$}}{
		\Indp\lIf{$x<1$}{$k\leftarrow\lceil \frac{1}{|x_i|w}\rceil$}
		\lElse{$k\leftarrow 1$}\;
		\lIf{$|y_{A^k,i}|\leq |y_{A^{-k},i}|$}{$P\leftarrow A^{k}\circ P$, \quad $W\leftarrow \concat(A^k,W)$}\;
		\lElse{$P\leftarrow A^{-k}\circ P$, \quad $W\leftarrow \concat(A^{-k},W)$}
	}\;
	\Indm
	\lElseIf{$|x_i|\geq|y_i|$ \tcp*{Case $4$}}{
		\Indp\lIf{$y<1$}{$l\leftarrow\lceil \frac{1}{|y_i|(w-1)}\rceil$}
		\lElse{$l\leftarrow 1$}\;
		\lIf{$|x_{B^l,i}|\leq |x_{B^{-l},i}|$}{$P\leftarrow B^{l}\circ P$,\quad $W\leftarrow \concat(B^l,W)$}\;
		\lElse{$P\leftarrow B^{-l}\circ P$, \quad $W\leftarrow \concat(B^{-l},W)$}
	}\;	
}
\Return{$W$}
\caption{Finding a word $W$ mapping $P$ to an element of $S$.}
\label{alg:Alg}
\end{algorithm}

Obviously, the conditions $|x_i|<|y_i|$ and $|x_i|\geq|y_i|$ of case $3$ and $4$, respectively, guarantee that all points are covered by the algorithm. To see that the algorithm terminates and thus yields a word $W$ with $W\circ P\in S$ we consider the values $|x_i|$, $|y_i|$ and $m\coloneqq \max (|x_i|,|y_i|)$ and will prove that in all cases -- with one exception -- $m$ is reduced. The exception is in case $4$, if $|x_i|=|y_i|$: since $B$ does not change $y$, after applying $B^{\pm l}$ the value of $m$ will still be $|y_i|$, but then, in the next step $P$ will be in case $3$ or $2$ and thus $m$ will be reduced. Hence after finitely many steps we get a point in the orbit of $P$, which is in $S$.

In the following we will often use two inequalities that hold for every point in $L_8$:
\begin{align}
\label{xAbsch}0\leq  x_r+x_i\om & < 1+\om\\
\label{yAbsch}0\leq  y_r+y_i\om & < \om
\end{align}

\paragraph{Case 1:} Let us begin with the case that \emph{$P\notin S$ is periodic under $B$}. As seen in \autoref{subsect:periodic}, in this case for the \ip \ $y_i$ the inequalities $0\leq y_i<1$ hold. Since $P$ is not in $S$, the absolute value of $x_i$ is at least $\Grenze$ and $m=|x_i|$.

The algorithm tells us to compute $A\inv B\inv A\circ P$. We will do this step by step, beginning with $A\circ P$:
\[A\circ P=\PointL{x}{y_r+2x_i-t_1+(y_i+x_r-t_2)\om} \textnormal{ with } t_1\in\{0,1,2\} \textnormal{ and } t_2\in\{0,1\}.\]
The $y$-coordinate of $B\inv A\circ P$ is the same as for $A\circ P$ and the $x$-coordinate has
\begin{align*}
\textnormal{\rp \ } & x_r'= -x_r-2x_i-y_r-2y_i+t_1+2t_2+t_3 \textnormal{ with } t_3\in\{0,1,2\} \textnormal{ and}\\
\textnormal{\ip \ } & x_i'=-x_r-x_i-y_r-y_i+t_1+t_2+t_4 \textnormal{ with } t_4\in\{0,1\}.
\end{align*}
The last step -- applying $A\inv$ -- does not change the $x$-coordinate, but the $y$-coordinate: Indeed, if we write the point $A\inv B\inv A\circ P$ in the form $\PointL{x_r'+x_i'\om}{y_r'+y_i'\om}$, we get $x_r'$ and $x_i'$ as above and for $y$:
\begin{align*}
y_r'&=2x_r+4x_i+3y_r+2y_i-3t_1-2t_2-2t_4+t_5  \textnormal{ with } t_5\in\{0,1,2\} \textnormal{ and}\\
y_i'&=2x_r+2x_i+y_r+3y_i-t_1-3t_2-t_3+t_6  \textnormal{ with } t_6\in\{0,1\}.
\end{align*}

We have to show that $m':=\max(|x_i'|,|y_i'|)<m=\max(|x_i|,|y_i|)=|x_i|$ and we do this by showing:
\begin{itemize}
\item[a)] if $x_i>\Grenze$, then $|x_i'|<|x_i|$.
\item[b)] if $x_i<-(\Grenze)$, then $|x_i'|<|x_i|$.
\item[c)] if $x_i>\Grenze$, then $|y_i'|<|x_i|$.
\item[d)] if $x_i<-(\Grenze)$, then $|y_i'|<|x_i|$.
\end{itemize}
For $|x_i|>\Grenze$ we know $\sgn(x_r+x_i)=-\sgn(x_i)$. Hence $|x_r+x_i|=\sgn(x_i)(x_r+x_i)$.

\begin{itemize}
\item[a)] $x_i>\Grenze>\frac{4}{2-w}$: This implies $x_i>wx_i-x_i+4$ and with \eqref{xAbsch} \[x_i>-x_r-x_i+4.\] 

Furthermore, we know that the possible values of $-y_r-y_i+t_1+t_2+t_4$ lie between $0$ and $4$. Hence we can estimate
\[|x_i|=x_i>-x_r-x_i+4\geq |-x_r-x_i|+|-y_r-y_i+t_1+t_2+t_4|\geq |x_i'|.\]
\item[b)] $x_i<-(\Grenze)<-\frac{5+w}{2-w}$: This inequality together with \eqref{xAbsch} implies \[-x_i>-x_iw+1+w + x_i+4>x_r+x_i+4\] and
\[|x_i|=-x_i>|-x_r-x_i|+|-y_r-y_i+t_1+t_2+t_4|\geq |x_i'|.\]

\item[c)] $x_i>\Grenze>\frac{7}{3-2w}$: Together with \eqref{xAbsch} this inequality implies \[x_i>2x_iw-2x_i+7\geq -2x_r-2x_i+7.\] Since the value of $y_r+3y_i-t_1-3t_2-t_3+t_6$ is between $-7$ and $4$ we can estimate
\[|x_i|=x_i>|2x_r+2x_i|+|y_r+3y_i-t_1-3t_2-t_3+t_6|\geq |y_i'|.\]

\item[d)] $x_i<-(\Grenze)=-\frac{9+2w}{3-2w}$: Together with \eqref{xAbsch} this implies \[-x_i>2+2w-2x_iw+2x_i+7>2x_r+2x_i+7\] and
\[|x_i|=-x_i>|2x_r+2x_i|+|y_r+3y_i-t_1-3t_2-t_3+t_6|\geq |y_i'|.\]
\end{itemize}

This completes case $1$.

\paragraph{Case 2:} Now \emph{$P\notin S$ is periodic under $A$} and we have to compare $P$ with $B\inv A\inv B \circ P$. The latter is of the form

\begin{align*}
x_r'=& 3x_r+2x_i+4y_r+6y_i-3t_1-2t_2-t_3-2t_4+t_5 \\
x_i'=&x_r+3x_i+3y_r+4y_i-t_1-3t_2-t_3-t_4+t_6 \\
y_r'=&-2x_i-y_r-2y_i+2t_2+t_3 \\
y_i'=&-x_r-y_r-y_i+t_1+t_4\\
\textnormal{with } & t_1,t_3,t_5\in\{0,1,2\} \textnormal{ and } t_2,t_4,t_6 \in \{0,1\}
\end{align*}

Since in this case all computations are very similar to the previous case -- just use \autoref{yAbsch} instead of \autoref{xAbsch} -- we will just give the bounds, that guarantee $m'<m$. Note that points periodic under $A$ have $|x_i|<1$ and thus $m=|y_i|$. Furthermore we use $\sgn(3y_r+4y_i)=-\sgn(y_i)=\sgn(y_r+y_i)$.

\begin{itemize}
\item[a)] if $y_i>\Grenze>\frac{8}{5-3w}$, then $|x_i'|<|y_i|$.
\item[b)] if $y_i<-(\Grenze)<-\frac{8+3w}{5-3w}$, then $|x_i'|<|y_i|$.
\item[c)] if $y_i>\Grenze>\frac{3}{2-w}$, then $|y_i'|<|y_i|$.
\item[d)] if $y_i<-(\Grenze)<-\frac{3+w}{2-w}$, then $|y_i'|<|y_i|$.
\end{itemize}

\paragraph{Case 3:} Now $|x_i|<|y_i|$ and we want to show that the power of $A$ given by the algorithm reduces $|y_i|$ and thus reduces $m$.

Remember that the difference between the \ip s of the $y$-coordinate of $P$ and of $A^k\circ P$
 is denoted by $\Delta_{A^k}(P)$ and by \eqref{eqn:DeltaAk} 
 this difference is
\begin{equation*}
\Delta_{A^k}(P)=\begin{cases}
-kx_iw -r \text{ for an } r\in (-1,1) & \text{ if } x\leq 1\\
k(x_r - 1) & \textnormal{ if } x>1.
\end{cases}
\end{equation*}

\begin{itemize}
\item If $x\leq 1$ and $k=\lceil \frac{1}{|x_i|w}\rceil=\frac{1}{|x_i|w}+s$ for an $s\in (0,1)$ we get \[\Delta_{A^k}(P)=-\sgn(x_i)-sx_iw-r_+ \quad \textnormal{and} \quad \Delta_{A^{-k}}(P)=\sgn(x_i)+sx_iw-r_-.\] Since $|r|<1$, the signs of $\Delta_{A^k}(P)$ and $\Delta_{A^{-k}}(P)$ are distinct: one is positive, the other is negative. Moreover, $|\Delta_{A^{\pm k}}(P)|\leq 1+s|x_i|w+1<2|y_i|$. Thus either $A^k$ or $A^{-k}$ reduces the absolute value $|y_i|$ and so $m$.
\item If $x>1$ and $k=1$, we have $\Delta_{A^{\pm k}}(P)=\pm(x_r-1)$. Since $P$ is not periodic under the action of $A$, the \rp \ $x_r\neq 1$. Moreover, \[|\Delta_{A^{\pm k}}(P)|\leq |x_r|+1<1+w+w|x_i|+1<2|y_i|.\] Thus either $A$ or $A^{-1}$ reduces the absolute value $|y_i|$ and so $m$.
\end{itemize}

\paragraph{Case 4:} Now $|y_i|\leq |x_i|$ and we have to show that the power of $B$ given by the algorithm reduces $|x_i|$. Remember \autoref{eqn:DeltaBl}:
\begin{equation*}
\Delta_{B^l}(P)=\begin{cases}
ly_i(1-\om) -r \text{ for an } r\in (-1,1) & \text{ if } y\leq 1\\
l (y_r + y_i - 1) & \textnormal{ if } y>1.
\end{cases}
\end{equation*}

\begin{itemize}
\item If $y\leq 1$ and $l=\lceil \frac{1}{|y_i|(w-1)}\rceil=\frac{1}{|y_i|(w-1)}+s$ for an $s\in (0,1)$ we get \[\Delta_{B^l}(P)=-\sgn(y_i)-sy_i(w-1)-r_+ \quad \textnormal{and} \quad \Delta_{B^{-l}}(P)=\sgn(y_i)+sy_i(w-1)-r_-.\] Since $|r|<1$, the signs of $\Delta_{B^l}(P)$ and $\Delta_{B^{-l}}(P)$ are distinct. Moreover, \[|\Delta_{B^{\pm l}}(P)|\leq 1+s|y_i|(w-1)+1<2|y_i|\leq 2 |x_i|.\] Thus either $B^l$ or $B^{-l}$ reduces $|x_i|$.
\item If $y>1$ and $l=1$, we have $\Delta_{B^{\pm l}}(P)=\pm(y_r+y_i-1)$. Since $P$ is not periodic under $B$, the change $\Delta_{B^{\pm l}}(P)$ is not $0$. Moreover, $y_r$ and $y_i$ have different signs and $|y_i|<|y_r|$ unless both $|y_r|$ and $|y_i|$ are 
 much smaller than $\Grenze$. Hence $|\Delta_{B^{\pm l}}(P)| < 2|x_i|$. Thus either $B$ or $B^{-1}$ reduces $|x_i|$.
\end{itemize}

This completes the proof, since in all cases $m=\max(|x_i|,|y_i|)$ is reduced: in the case $|x_i|=|y_i|$ after two steps, otherwise in each step. This finally leads to a point with $|x_i|,|y_i|<\Grenze$, which is in the finite set $S$.
\end{proof}

\begin{rem}
Points in $S$ can have approximately $((\Grenze)\cdot 2N)^2$ different values for $x_i$ and $y_i$ and for each such pair there might be up to $((1+w)\cdot N)\cdot(w\cdot N)$ possible values for $x_r$ and $y_r$. Thus the upper bound on the number of orbits with fixed $N$ given by the theorem is about $(\Grenze)^2 (w+2) N^4=(8114+5737w) N^4\approx 16227 N^4$. This is a very bad upper bound, supposably the correct number of orbits is very close (or equal) to the lower bound given by the number of connected components of the graph described in the beginning of this section. In particular for $N=1$ with a little bit more work one can show that, indeed, all non-periodic points are in the same orbit. This means that for all non-periodic points $P=\PointL{x_r+x_iw}{y_r+y_iw}$ with $x_r,x_i,y_r,y_i\in\ZZ$ the groups $\SL(L_8;P)$ are conjugate and have the same critical exponent.
\end{rem}

\printbibliography

\end{document}